\documentclass[12pt,reqno]{amsart}
\usepackage{amsfonts}

\usepackage{amssymb,amsmath,graphicx,amsfonts,euscript}
\usepackage{color}

\newtheorem{thm}{Theorem}[section]

\newtheorem{Pros}[thm]{Proposition}
\newtheorem{define}[thm]{Definition}
\newtheorem{rem}[thm]{Remark}

\newtheorem{lemma}[thm]{Lemma}

\allowdisplaybreaks \voffset=-0.2in \numberwithin{equation}{section}

\begin{document}
\bigskip

\centerline{\Large\bf  Global well-posedness of the 2D Boussinesq
equations}
\smallskip

\centerline{\Large\bf  with fractional Laplacian dissipation}

\bigskip

\centerline{Zhuan Ye,\,\, Xiaojing Xu}

\bigskip

\centerline{School of Mathematical Sciences, Beijing Normal University,}
\medskip

\centerline{Laboratory of Mathematics and Complex Systems, Ministry of Education,}
\medskip

\centerline{Beijing 100875, People's Republic of China}

\medskip

\centerline{Ye's E-mail: \texttt{yezhuan815@126.com
}}
\centerline{Xu's E-mail: \texttt{xjxu@bnu.edu.cn}}

\bigskip
{\bf Abstract:}~~%
As a continuation of the previous work \cite{YX2015}, in this paper we focus on the Cauchy problem of the two-dimensional
(2D) incompressible Boussinesq equations with fractional Laplacian
dissipation. We give an elementary proof of the global regularity of
the smooth solutions of the 2D Boussinesq equations with a new range
of fractional powers of the Laplacian. The argument is based on the nonlinear lower bounds for the fractional Laplacian established in \cite{CV}. Consequently, this result significantly improves the
recent works \cite{CV,YJW,YX2015}.

{\vskip 1mm
 {\bf AMS Subject Classification 2010:}\quad 35Q35; 35B65; 76D03.

 }

\vskip .3in
\section{Introduction}\label{intro}
In this paper, we are interested in studying the following 2D incompressible Boussinesq equations with fractional Laplacian dissipation
\begin{equation}\label{Bouss}
\left\{\aligned
&\partial_{t}u+(u \cdot \nabla) u+\nu\Lambda^{\alpha}u+\nabla p=\theta e_{2},\,\,\,\,\,x\in \mathbb{R}^{2},\,\,t>0, \\
&\partial_{t}\theta+(u \cdot \nabla) \theta+\kappa\Lambda^{\beta}\theta=0, \,\,\,\qquad\qquad x\in \mathbb{R}^{2},\,\,t>0,\\
&\nabla\cdot u=0, \,\,\,\,\,\qquad \qquad \qquad \qquad \qquad \quad x\in \mathbb{R}^{2},\,\,t>0,\\
&u(x, 0)=u_{0}(x),  \quad \theta(x,0)=\theta_{0}(x),\,\,\,\quad x\in \mathbb{R}^{2},
\endaligned\right.
\end{equation}
where the numbers $\nu\geq0$, $\kappa\geq0$, $\alpha\in [0,\,2]$ and $\beta\in [0,\,2]$ are real parameters. Here $u(x,\,t)=(u_{1}(x,\,t),\,u_{2}(x,\,t))$ is a vector field denoting
the velocity, $\theta=\theta(x,\,t)$ is a scalar function denoting
the temperature, $p$ is the scalar pressure and
$e_{2}=(0,\,1)$. The fractional Laplacian operator $\Lambda^{\alpha}$,
$\Lambda:=(-\Delta)^{\frac{1}{2}}$ denotes the Zygmund operator which is defined through the Fourier transform, namely $$\widehat{\Lambda^{\alpha}
f}(\xi)=|\xi|^{\alpha}\hat{f}(\xi),$$
where
$$\hat{f}(\xi)=\frac{1}{(2\pi)^{2}}\int_{\mathbb{{R}}^{2}}{e^{-ix\cdot\xi}f(x)\,dx}.$$
The fractional dissipation operator severs to model many physical phenomena (see \cite{DImbert}) in hydrodynamics and molecular biology such as anomalous diffusion in semiconductor growth (see \cite{PSW}).
We remark the convention that by $\alpha=0$ we mean that there is no dissipation in
$(\ref{Bouss})_{1}$, and similarly $\beta=0$ represents that there is no dissipation in
$(\ref{Bouss})_{2}$.

The standard Boussinesq equations (namely $\alpha=\beta=2$) are of relevance to study a number of models coming from atmospheric or oceanographic turbulence
where rotation and stratification play an important role (see for example \cite{MB,PG}). Moreover, as point out in \cite{MB}, the 2D inviscid Boussinesq equations, namely (\ref{Bouss}) with $\alpha=\beta=0$
are identical to the incompressible axi-symmetric (away from the $z$-axis) swirling 3D Euler equations.
There are geophysical circumstances in which the Boussinesq
equations with fractional Laplacian may arise. The effect of kinematic and thermal diffusion is attenuated by the thinning of atmosphere. This anomalous
attenuation can be modeled by using the space fractional Laplacian (see \cite{Cap,Gill}).

The global well-posedness of the 2D Boussinesq equations has
recently drawn a lot of attention  and many important results have
been established. It is well-known that the system (\ref{Bouss})
with full Laplacian dissipation (namely, $\alpha=\beta=2$) is global
well-posed, see, e.g., \cite{Can}.  In the case of inviscid
Boussinesq equations, the global regularity problem turns out to be
extremely difficult and remains outstandingly open. Therefore, it is
natural to consider the intermediate cases. Actually, many important
progress has recently been made on this direction. Almost at the
same time, Chae \cite{C1} and Hou and Li \cite{HL} proved the global
regularity for the system (\ref{Bouss}) when $\alpha=2$ and
$\beta=0$ or $\alpha=0$ and $\beta=2$ independently. Since then, much efforts
are devoted to the global regularity of (\ref{Bouss}) with the
smallest possible $\alpha\in (0,2)$ and $\beta\in (0,2)$. As pointed
out in \cite{JMWZ}, we can classify $\alpha$ and $\beta$ into three
categories: the subcritical case when $\alpha +\beta>1$, the
critical case when $\alpha +\beta=1$ and the supercritical case when
$\alpha +\beta<1$.  As a rule of thumb, with current methods it seems impossible to obtain
the global regularity for the 2D Boussinesq equations with supercritical dissipation.
Recently, Jiu, Wu and Yang \cite{JWYang} established the eventual
regularity of weak solutions of the system (\ref{Bouss}) when
$\alpha$ and $\beta$ are in the suitable supercritical range.
For the critical case, there are several works are available.
In the two elegant papers, Hmidi, Keraani and
Rousset \cite{HK3,HK4} established the global well-posedness result
to the system (\ref{Bouss}) with two special critical cases, namely
$\alpha=1$ and $\beta=0$ or $\alpha=0$ and $\beta=1$. The more general
critical case, that is $\alpha+\beta=1$ with $0<\alpha,\,\beta<1$ is
extremely difficult. The standard energy estimates do not yield the
global bounds in any Sobolev spaces when $\alpha$ and $\beta$ in the
critical case. Very recently,  the global regularity of the general
critical case $\alpha+\beta=1$ with
$\alpha>\frac{23-\sqrt{145}}{12}\thickapprox 0.9132$ and $0<\beta<1$
was recently examined by Jiu, Miao, Wu and Zhang \cite{JMWZ}. This
result was further improved by Stefanov and Wu \cite{SW} by further
enlarging the range of $\alpha$ with $\alpha+\beta=1$ and
$1>\alpha>\frac{\sqrt{1777}-23}{24}\thickapprox 0.7981$ and
$0<\beta<1$. Here we want to state that even in the subcritical
ranges, namely $\alpha+\beta>1$ with $0<\alpha<1$ and $0<\beta<1$,
the global regularity of (\ref{Bouss}) is also definitely nontrivial
and quite difficult. In fact, to the best of our knowledge there are
only several works concerning the subcritical cases. More precisely,
 Miao and Xue \cite{MX} obtained
the global regularity for system (\ref{Bouss}) for the case $\nu>0$, $\kappa>0$ and
$$\frac{6-\sqrt{6}}{4}<\alpha<1,\,\,\,1-\alpha<\beta<\min\Big\{\frac{7+2\sqrt{6}}{5}\alpha-2,\,\,
\frac{\alpha(1-\alpha)}{\sqrt{6}-2\alpha},\,\,\,2-2\alpha \Big\}.$$
In addition, Constantin and Vicol \cite{CV} verified the global regularity of the system (\ref{Bouss}) on the case when the thermal diffusion dominates, namely
$$\nu>0,\,\,\,\kappa>0,\,\,\,0<\alpha<2,\,\,\,0<\beta<2,\,\,\,\beta>\frac{2}{2+\alpha}.$$
Recently, Yang, Jiu and Wu \cite{YJW} proved the global well-posedness of the system (\ref{Bouss}) with
$$\nu>0,\,\,\,\kappa>0,\,\,\,0<\alpha <1,\,\,\,0< \beta<1,\,\,\,\beta>1-\frac{\alpha}{2},\,\,\,
\beta\geq\frac{2+\alpha}{3},\,\,\,\beta>\frac{10-5\alpha}{10-4\alpha}.$$
Very recently, the authors \cite{YXX} established  the global regularity for the 2D Boussinesq equations with a new range of fractional powers, namely $\nu>0,\,\kappa>0$ and
$$0.783\thickapprox\frac{21-\sqrt{217}}{8}<\alpha<1,\quad 1-\alpha<\beta<
\min\Big\{\frac{\alpha}{2},\,\,
\frac{(3\alpha-2)(\alpha+2)}{10-7\alpha},
\,\,\frac{2-2\alpha}{4\alpha-3}\Big\}.$$
Here we also want to mention that the two works \cite{CV,YJW} have been improved by the recent manuscript \cite{YX2015}. More precisely, the authors in \cite{YX2015} established the global regularity result for the 2D Boussinesq equations with
$$\nu>0,\,\,\,\kappa>0,\,\,\,0<\alpha <1,\,\,\,0< \beta<1,\,\,\,\beta>1-\frac{\alpha}{2},\,\,\,
\beta\geq\frac{2+\alpha}{3}.$$

The case of partial anisotropic dissipation has been considered in
several settings (see for instance \cite{ACSWXY,CW,DP3,LT,LLT}). For
the global smooth solutions to the damped Boussinesq equations with
small initial datum, we refer the readers to the recent works
\cite{ACWX,WXY}. Moreover, the global unique
solution of the Boussinesq equations for the Yudovich type data has
been established by many works, and we refer the readers to the
interesting works \cite{DP2,WuXue2012,XX2014,WuXu}. It is worth
remarking that there are several works concerning the global
regularity for the 2D Boussinesq equations with logarithmical dissipation (see, e.g., \cite{Hmidi2011,CW2,KRTW}). Many other interesting recent results on the Boussinesq equations can be found, with no intention to be complete (see, e.g.,
\cite{CKN,D,JMWZ,LaiPan,LLT,LWZ,XX,YZ} and the references therein).

The goal of this paper is to establish the global regularity of
solutions to the system (\ref{Bouss}) with a new range of fractional
powers of the Laplacian. Since the concrete values of the constant
$\nu,\,\kappa$ play no role in our discussion, for this reason, we
shall assume $\nu=\kappa=1$ throughout this paper. Now let us state
our main result as follows
\begin{thm}\label{Th1} Let $0<\alpha<1$ and $0<\beta<1$ satisfy
\begin{equation}\label{Condition}
\beta>\beta^{\ast}:=\left\{\aligned
&\max\Big\{\frac{2}{3},\,\,\frac{4-\alpha^{2}}{4+3\alpha}\Big\},\, \,\,\quad\qquad\qquad\quad 0<\alpha\leq \frac{2}{3},\\
&\frac{2-\alpha}{2},\qquad\qquad\qquad\qquad\qquad\qquad
\frac{2}{3}\leq \alpha<1.\\
\endaligned\right.
\end{equation}
Assume that $(u_{0}, \theta_{0})
\in H^{s}(\mathbb{R}^{2})\times H^{s}(\mathbb{R}^{2})$ for any $s>2$ and satisfies $\nabla\cdot u_{0}=0$.
Then the system (\ref{Bouss}) admits a unique global solution such that for any
given $T>0$
$$u\in C([0, T]; H^{s}(\mathbb{R}^{2}))\cap L^{2}([0, T]; H^{s+\frac{\alpha}{2}}(\mathbb{R}^{2})),$$
$$\theta\in C([0, T]; H^{s}(\mathbb{R}^{2}))\cap L^{2}([0, T];
H^{s+\frac{\beta}{2}}(\mathbb{R}^{2})).$$
\end{thm}

\vskip .1in
Let us give some remarks about our result.
\begin{rem}\rm
On the one hand, one can easily check that
$\frac{2}{2+\alpha}>\beta^{\ast}$, thus Theorem \ref{Th1}
significantly improves Theorem 6.1 of \cite{CV}. On the other hand,
our theorem also significantly improves Theorem 1.1 of \cite{YJW},
which obtained the global regularity result under the assumption
$$\beta>\max\Big\{\frac{2-\alpha}{2},\,\,\frac{2+\alpha}{3},\,\,
\frac{10-5\alpha} {10-4\alpha}\Big\}>\beta^{\ast}.$$ Finally, we
\cite{YX2015} have proved that the system (\ref{Bouss}) admits a
unique global solution provided
$$\beta>\frac{2-\alpha}{2}\quad \rm{and}\quad \beta\geq\frac{2+\alpha}{3}.$$
Obviously, Theorem \ref{Th1} significantly improves the result of
\cite{YX2015}.
\end{rem}
\begin{rem}\rm
Through the proof, we find that Theorem \ref{Th1} is always true for
$\beta>\frac{2-\alpha}{2}$ with any $0<\alpha<1$. However, it is
easy to check that when $0<\alpha<\frac{2}{3}$, it holds
$$\max\Big\{\frac{2}{3},\,\,\frac{4-\alpha^{2}}{4+3\alpha}\Big\}<\frac{2-\alpha}{2}.
$$
In fact, the proof of \textbf{Case 1} (below) is much complicated
than the proof of the case $\beta>\frac{2-\alpha}{2}$ with any
$0<\alpha<1$.
\end{rem}

\begin{rem}\rm
Through the proof of Theorem \ref{Th1}, we strongly believe that if
one may establish Lemma \ref{AZL302} under somewhat weaker
conditions than $\beta>\beta^{\ast}$, then Theorem \ref{Th1} can
also be improved. As suggested by Jiu, Miao, Wu and Zhang in
\cite{JMWZ}, the expected subcritical result is $\beta>1-\alpha$
($1-\alpha<\beta^{\ast}$). However, at the moment we are not able to
weaken the conditions $\beta>\beta^{\ast}$. We will investigate this issue
further in our future work.
\end{rem}
\begin{rem}\rm
The nonlinear lower bounds for the fractional Laplacian \cite{CV} or the
H$\rm \ddot{o}$lder estimates for advection
fractional-diffusion equations \cite{Sil20121} entails us that if one can show that for any given $T>0$
$$\sup_{0\leq t\leq T}\|u(t)\|_{C^{\alpha}}<\infty\quad \mbox{or}\quad
\sup_{0\leq t\leq T}\|\omega(t)\|_{L^{\frac{2}{1-\alpha}}}<\infty,$$
under the assumption $\beta>1-\alpha$, then the equations are well-posed in the smooth category up to time $T$. Here $w:=\nabla\times u=\partial_{x_{1}}u_{2}-\partial_{x_{2}}u_{1}$ is the vorticity and $C^{\alpha}$ stands for the classical H$\rm \ddot{o}$lder space.
\end{rem}

We outline the main idea in the proof of this theorem. A large
portion of the efforts are devoted to obtaining global {\it a
priori} bounds for $u$ and $\theta$ on the interval $[0,\,T]$.
According to the definition of $\beta^{\ast}$, the proof of Theorem
\ref{Th1} is divided into two cases, that is,
\begin{eqnarray}
&&\mbox{\textbf{Case 1}}:
0<\alpha\leq\frac{2}{3},\,\,\,\,\max\Big\{\frac{2}{3},\,\,\frac{4-\alpha^{2}}{4+3\alpha}\Big\}<\beta<1,\nonumber\\&& \mbox{\textbf{Case
2}}:
\frac{2}{3}\leq\alpha<1,\,\,\,\,\frac{2-\alpha}{2}<\beta<1.\nonumber
\end{eqnarray}
To start, let us say some words about the proof of the work \cite{YJW}, where the main idea of the work \cite{YJW} is to consider the combined quantity $G$ (see (\ref{t305}) for more details)
\begin{eqnarray}\label{Comb}\partial_{t}G+(u\cdot\nabla)
G+\Lambda^{\alpha}G=-[\mathcal {R}_{\beta},\,u\cdot\nabla]\theta
+\Lambda^{\alpha-\beta}\partial_{x_{1}}\theta.\end{eqnarray}
Here and in sequel, we have used the standard commutator notation
$$[\mathcal {R}_{\beta},\,u\cdot\nabla]\theta:=
\mathcal {R}_{\beta}(u\cdot\nabla\theta)-u\cdot\nabla\mathcal
{R}_{\beta}\theta.$$
Invoking some commutator estimates and some computations, the combined quantity $G$ satisfies
$$\sup_{0\leq t\leq T}\|G(t)\|_{L^{2}}^{2}+\int_{0}^{T}{\|\Lambda^{\frac{\alpha}{2}}G(\tau)\|_{L^{2}}^{2}
\,d\tau}<\infty,$$
which is true for $\beta>\frac{2-\alpha}{2} $
and $ \beta\geq\frac{\alpha+2}{3}$. \\
Then they show the estimate $$\sup_{0\leq t\leq T}\|G(t)\|_{L^{p_{0}}}<\infty$$ for some $p_{0}>2$. This estimate together with the iterative process leads to
$$\sup_{0\leq t\leq T}\|G(t)\|_{L^{p}}<\infty,\quad \mbox{for any}\,\,p_{0}\leq p<\infty,$$
which is valid for $\beta>1-\frac{\alpha}{2},\,\,\,
\beta\geq\frac{2+\alpha}{3}$ and
$\beta>\frac{10-5\alpha}{10-4\alpha}$.

However, the main argument used here is completely different from
the work \cite{YJW}. For \textbf{Case 1}, in view of several
commutator estimates, we can show by combining $L^{2}$-norm of the
combined quantity $G$ and the temperature $\theta$
\begin{eqnarray}\label{K01}\sup_{0\leq t\leq T}(\|G\|_{L^{2}}^{2}+\|\Lambda^{\varrho}\theta\|_{L^{2}}^{2})(t)
+\int_{0}^{T}{\big(\|\Lambda^{\frac{\alpha}{2}}G\|_{L^{2}}^{2}+\|\Lambda^{\varrho+
\frac{\beta}{2}}\theta\|_{L^{2}}^{2}\big)(\tau)\,d\tau}<\infty.\end{eqnarray}

whenever $0\leq\varrho<\frac{\beta}{2}$.

For \textbf{Case 2}, by combining $L^{2}$-norm of the vorticity
$\omega$ and the temperature $\theta$, one can conclude that
\begin{eqnarray}\label{AK01}\sup_{0\leq t\leq T}(\|\omega\|_{L^{2}}^{2}
+\|\Lambda^{\delta}\theta\|_{L^{2}}^{2})(t)
+\int_{0}^{T}{\big(\|\Lambda^{\frac{\alpha}{2}}\omega\|_{L^{2}}^{2}
+\|\Lambda^{\delta+
\frac{\beta}{2}}\theta\|_{L^{2}}^{2}\big)(\tau)\,d\tau}<\infty.\end{eqnarray}
whenever $0\leq\delta<\frac{\beta}{2}$.

The above two bounds (\ref{K01}) and (\ref{AK01}) are the key
component of this paper. With the help of the two bounds (\ref{K01})
and (\ref{AK01}), we will establish the following key global bound
$$\max_{0\leq t\leq T}\|u(t)\|_{L^{r}}<\infty$$ for
any $2\leq r<\infty$. Thanks to the nonlinear lower bounds for the
fractional Laplacian established in \cite{CV}, the following key
estimate holds
$$\max_{0\leq t\leq T}\|\nabla\theta (t)\|_{L^{\infty}}<\infty.$$
Finally, with the above estimate at our disposal, the global
regularity of $u$ and $\theta$ following a standard approach (see
for instance \cite{CKN,D,MB}).

\vskip .2in
The rest of the paper is organized as follows. In Section 2, we
obtain the {\it a priori} estimates for sufficiently smooth
solutions of the system (\ref{Bouss}). Section 3 is devoted to the
proof of Theorem \ref{Th1}. Finally, in the Appendix, we give the
proof of Lemmas \ref{NCE} and \ref{Lem23} for the sake of
completeness.

\vskip .4in
\section{ A priori estimates}\setcounter{equation}{0}
This section is devoted to the {\it a priori} estimates which can be viewed as a preparation for the proof of Theorem \ref{Th1}.
To simplify the notations, we shall use the letter $C$ to denote a generic constant which may vary from line to line. The dependence of $C$ on other parameters is usually clear from the context and we shall explicitly specify it whenever necessary.

The first lemma concerns the following commutator estimate, which
plays a key role in proving our main result. The proof can be performed by making use of the
Littlewood-Paley technique. To facilitate the reader, we will sketch the proof in
the Appendix.
\begin{lemma}\label{NCE}
Let $\frac{1}{p}=\frac{1}{p_{1}}+\frac{1}{p_{2}}$ with $p\in[2,
\infty)$ and $p_{1},\,p_{2}\in[2, \infty]$. Assume $r\in[1,
\infty]$, $\delta\in(0,1)$, $s\in(0, 1)$ such that $s+\delta<1$,
then it holds
\begin{align}\label{TtNCE}
\|[\Lambda^{\delta},f]g\|_{B_{p,r}^{s}}\leq
C(p,r,\delta,s)\big(\|\nabla f\|_{L^{p_{1}}}\|g\|_{
{B}_{p_{2},r}^{s+\delta-1}}+\|f\|_{L^{2}}\|g\|_{L^{2}}\big).
\end{align}
In particular,
\begin{align}\label{tNCE}
\|[\Lambda^{\delta},f]g\|_{B_{p,r}^{s}}\leq
C(p,r,\delta,s)\big(\|\nabla f\|_{L^{p}}\|g\|_{
{B}_{\infty,r}^{s+\delta-1}}+\|f\|_{L^{2}}\|g\|_{L^{2}}\big).
\end{align}
Here and in what follows, $B_{p,r}^{s}$ stands for the classical
Besov space (see appendix for its precise definition).
\end{lemma}

In order to prove \textbf{Case 2}, we shall use the next two commutator
estimates involving $\mathcal
{R}_{\beta}:=\partial_{x_{1}}\Lambda^{-\beta}$.
\begin{lemma}[see \cite{SW}]\label{Lem22}
Assume that $\frac{1}{2}<\beta<1$ and $1<p_{2}<\infty$,
$1<p_{1},\,p_{3}\leq\infty$ with
$\frac{1}{p_{1}}+\frac{1}{p_{2}}+\frac{1}{p_{3}}=1$. Then for $0\leq
s_{1}<1-\beta$ and $s_{1}+s_{2}>1-\beta$, the following holds true
\begin{align}\label{t202}
\Big|\int_{\mathbb{R}^{2}}{ F[\mathcal
{R}_{\beta},\,u_{G}\cdot\nabla]\theta\,dx}\Big|\leq
C\|\Lambda^{s_{1}}\theta\|_{L^{p_{1}}}\|F\|_{W^{s_{2},\,p_{2}}}
\|G\|_{L^{p_{3}}},\end{align} where
$u_{G}:=\nabla^{\perp}\Delta^{-1}G$ and $W^{s,\,p}$ denotes the
standard Sobolev space.
\end{lemma}

\begin{lemma}[see \cite{LSWXY}]\label{Lem23}
Let $\frac{1}{p_{1}}+\frac{1}{p_{2}}=\frac{1}{p}$ for any $2\leq
p_{1},\,p_{2}\leq\infty$ and $2\leq p<\infty$. Assume that
$0<\beta<2$ and $\nabla\cdot u=0$, then we have for any $r\in[1,\,\infty]$
\begin{align}\label{t203}
\|[\mathcal {R}_{\beta},\,u\cdot\nabla]\theta\|_{L^{p}}\leq
C(\|\nabla
u\|_{L^{p_{1}}}\|\theta\|_{B_{p_{2},1}^{1-\beta}}+
\|u\|_{L^{r}}\|\theta\|_{L^{2}}).\end{align}
\end{lemma}
For the sake of completeness, we will give the proof of Lemma
\ref{Lem23} in the Appendix.
\\

Finally, let us recall the following fractional type
Gagliardo-Nirenberg inequality which is due to
Hajaiej-Molinet-Ozawa-Wang \cite{HMOW}.
\begin{lemma}\label{Lem24}
Let $0<p,\,p_{0},\,p_{1},\,q,\,q_{0},\,q_{1}\leq\infty$,
$s,\,s_{0},\,s_{1}\in \mathbb{R}$ and $0\leq \vartheta\leq1$. Then
the following fractional type Gagliardo-Nirenberg inequality
\begin{align}\label{t204}
\|v\|_{\dot{B}_{p,q}^{s}(\mathbb{R}^{n})}\leq
C\|v\|_{\dot{B}_{p_{0},q_{0}}^{s_{0}}(\mathbb{R}^{n})}^{1-\vartheta}
\|v\|_{\dot{B}_{p_{1},q_{1}}^{s_{1}}(\mathbb{R}^{n})}^{\vartheta}\end{align}
holds for all $v\in \dot{B}_{p_{0},q_{0}}^{s_{0}}\cap
\dot{B}_{p_{1},q_{1}}^{s_{1}}$ if and only if
$$\frac{n}{p}-s=(1-\vartheta)\big(\frac{n}{p_{0}}-s_{0}\big)+\vartheta\big(\frac{n}{p_{1}}
-s_{1}\big),\qquad s\leq (1-\vartheta)s_{0}+\vartheta s_{1},$$
$$\frac{1}{q}\leq \frac{1-\vartheta}{q_{0}}+\frac{\vartheta}{q_{1}},\quad \mbox{  if}\,\,\,p_{0}\neq p_{1}\,\,\,\mbox{and}\,\,\,s=(1-\vartheta)s_{0}+\vartheta s_{1},$$
$$s_{0}\neq s_{1}\,\,\,\mbox{or}\,\,\,\frac{1}{q}\leq \frac{1-\vartheta}{q_{0}}+\frac{\vartheta}{q_{1}},\quad \mbox{  if}\,\,\,p_{0}= p_{1}\,\,\,\mbox{and}\,\,\,s=(1-\vartheta)s_{0}+\vartheta s_{1},$$
$$s_{0}-\frac{n}{p_{0}}\neq s-\frac{n}{p}\,\,\,\mbox{or}\,\,\,\frac{1}{q}\leq \frac{1-\vartheta}{q_{0}}+\frac{\vartheta}{q_{1}},\quad \mbox{  if}\,\,\,s<(1-\vartheta)s_{0}+\vartheta s_{1}.$$
A special consequence of (\ref{t204}) is the following bound
\begin{align}\label{t205}
\|v\|_{\dot{B}_{4,1}^{1-\beta}(\mathbb{R}^{2})}\leq
C\|v\|_{\dot{B}_{2,2}^{s}(\mathbb{R}^{2})}^{\lambda}
\|v\|_{\dot{B}_{\infty,\infty}^{0}(\mathbb{R}^{2})}^{1-\lambda},\quad
\lambda=\frac{2\beta-1}{2-2s},\end{align} where
$2-2\beta<s<\frac{3-2\beta}{2}$ with $\frac{1}{2}<\beta<1$.

We also have
\begin{eqnarray}\label{t206}
\|\Lambda^{\gamma\beta}v\|_{L^{\frac{1}{\gamma}}(\mathbb{R}^{2})}
\leq
C\|\Lambda^{\frac{\beta}{2}}v\|_{L^{2}(\mathbb{R}^{2})}^{2\gamma}
\|v\|_{L^{\infty}(\mathbb{R}^{2})}^{1-2\gamma},\quad \beta>0,\,\,\,
0<\gamma<\frac{1}{2}.
\end{eqnarray}
\end{lemma}
\begin{rem}\rm
Lemma \ref{Lem24} is also true in the nonhomogeneous framework.
\end{rem}

It follows from the basic energy estimates that the corresponding
solution $(u,\,\theta)$ of the system (\ref{Bouss}) obeys the
following global bounds.
\begin{lemma}\label{L301}
Assume $(u_{0},\,\theta_{0})$ satisfies the assumptions stated in Theorem \ref{Th1}.
Then the corresponding solution $(u, \theta)$
of (\ref{Bouss}) admits the following bounds for any $t>0$
\begin{eqnarray}\label{t301}
&&\|\theta(t)\|_{L^{2}}^{2}+\int_{0}^{t}{
\|\Lambda^{\frac{\beta}{2}}\theta(\tau)\|_{L^{2}}^{2}\,d\tau}\leq
\|\theta_{0}\|_{L^{2}}^{2},\nonumber\\&& \|\theta(t)\|_{L^{p}}\leq
\|\theta_{0}\|_{L^{p}},\quad \forall p\in [2, \infty],\nonumber\\&&
\|u(t)\|_{L^{2}}^{2}+\int_{0}^{t}{
\|\Lambda^{\frac{\alpha}{2}}u(\tau)\|_{L^{2}}^{2}\,d\tau}\leq
(\|u_{0}\|_{L^{2}}+t\|\theta_{0}\|_{L^{2}})^{2}.
\end{eqnarray}
\end{lemma}

\vskip .3in
\subsection{\textbf{Case 1}}
As well-known, when $0<\alpha,\,\beta<1$, it is impossible to obtain
the global $H^{1}$ bound of $(u,\,\theta)$ by direct energy estimate
method. Actually, applying operator $\mbox{curl}$ to the first
equation in $(\ref{Bouss})$, we have the following vorticity
$w=\nabla\times u$
\begin{eqnarray}\label{t303}
\partial_{t}w+(u\cdot\nabla)w+\Lambda^{\alpha} w=\partial_{x_{1}}\theta.
\end{eqnarray}
However, the "vortex stretching" term $\partial_{x_{1}}\theta$
appears to prevent us from proving any global bound for $w$. To
circumvent this difficulty, a natural idea would be to eliminate
$\partial_{x_{1}}\theta$ from the vorticity equation. To this end,
we generalize the idea of Hmidi, Keraani and Rousset \cite{HK3,HK4}
to introduce a new quantity. More precisely, we set the combined
quantity
$$G=\omega-\mathcal {R}_{\beta}\theta,\quad \mathcal {R}_{\beta}:=\partial_{x_{1}}\Lambda^{-\beta},$$
which obeys the following equation
\begin{eqnarray}\label{t305}\partial_{t}G+(u\cdot\nabla)
G+\Lambda^{\alpha}G=-[\mathcal {R}_{\beta},\,u\cdot\nabla]\theta
+\Lambda^{\alpha-\beta}\partial_{x_{1}}\theta.\end{eqnarray} Since
$u$ is determined by $\omega$ through the Biot-Savart law, we have
\begin{eqnarray}\label{t306}u=\nabla^{\perp}\Delta^{-1}\omega
=\nabla^{\perp}\Delta^{-1}(G+\mathcal
{R}_{\beta}\theta)=\nabla^{\perp}\Delta^{-1}G+\nabla^{\perp}\Delta^{-1}\mathcal
{R}_{\beta}\theta:=u_{G}+u_{\theta}.\end{eqnarray}

\vskip .1in
We are now in the position to derive the following estimates
concerning $G$ and $\theta$.
\begin{lemma}\label{AZL302}
Assume $(u_{0},\,\theta_{0})$ satisfies the assumptions stated in
Theorem \ref{Th1}. Let $(u, \theta)$ be the corresponding solution
of the system (\ref{Bouss}). If
$\beta>\max\Big\{\frac{2}{3},\,\,\frac{4-\alpha^{2}}{4+3\alpha}\Big\}$,
then the following estimate holds for any
$\max\Big\{ \frac{4-5\beta}{2},\,\,\frac{2+\alpha-3\beta}{2}\Big\}<\varrho<\frac{\beta}{2}$
and $t\in[0, T]$
\begin{eqnarray}\label{AZ001}
\|G(t)\|_{L^{2}}^{2}+\|\Lambda^{\varrho}\theta(t)\|_{L^{2}}^{2}
+\int_{0}^{t}{\big(\|\Lambda^{\frac{\alpha}{2}}G\|_{L^{2}}^{2}+\|\Lambda^{\varrho+
\frac{\beta}{2}}\theta\|_{L^{2}}^{2}\big)(\tau)\,d\tau}\leq
C(T,\,u_{0},\,\theta_{0}),
\end{eqnarray}
where $C(T,\,u_{0},\,\theta_{0})$ is a constant depending on $T$ and
the initial data.
\end{lemma}
\begin{rem}\rm
Although the above estimate (\ref{AZ001}) holds for $\max\Big\{\frac{4-5\beta}{2},\,\,
\frac{2+\alpha-3\beta}{2}\Big\}<\varrho<\frac{\beta}{2}$, yet
by energy estimate (\ref{t301}) and the classical interpolation, we
find that (\ref{AZ001}) is true for any
$0\leq\varrho<\frac{\beta}{2}$.
\end{rem}
\begin{proof}[\textbf{Proof of Lemma \ref{AZL302}}]
Applying $\Lambda^{\varrho}$ to $(\ref{Bouss})_{2}$ and taking the
inner product with $\Lambda^{\varrho}\theta$, we obtain
\begin{eqnarray}\label{AZ002}\frac{1}{2}\frac{d}{dt}\|\Lambda^{\varrho}\theta\|_{L^{2}}^{2}
+\|\Lambda^{\varrho+\frac{\beta}{2}}\theta\|_{L^{2}}^{2}=-\int_{\mathbb{R}^{2}}
\Lambda^{\varrho}\big(u \cdot
\nabla\theta\big)\Lambda^{\varrho}\theta\,dx.
\end{eqnarray}
Hence, an application of the divergence-free condition and
commutator estimate (\ref{tNCE}) directly yields
\begin{eqnarray}\label{Fi01}\Big|\int_{\mathbb{R}^{2}}
\Lambda^{\varrho}\big(u \cdot
\nabla\theta\big)\Lambda^{\varrho}\theta\,dx\Big|
&=&\Big|\int_{\mathbb{R}^{2}}
[\Lambda^{\varrho}, u \cdot \nabla]\theta\,\,\Lambda^{\varrho}\theta\,dx\Big|\nonumber\\
&=&\Big|\int_{\mathbb{R}^{2}}
\nabla\cdot[\Lambda^{\varrho}, u]\theta\,\,\Lambda^{\varrho}\theta\,dx\Big|\nonumber\\
&\leq& C\|\Lambda^{1-\frac{\beta}{2}}[\Lambda^{\varrho}, u]\theta\|_{L^{2}}\|\Lambda^{\varrho+\frac{\beta}{2}}\theta\|_{L^{2}}\nonumber\\
&\leq& C\|[\Lambda^{\varrho}, u ]\theta\|_{H^{1-\frac{\beta}{2}}}\|\Lambda^{\varrho+\frac{\beta}{2}}\theta\|_{L^{2}}\nonumber\\
&\leq& C\|[\Lambda^{\varrho}, u
]\theta\|_{B_{2,2}^{1-\frac{\beta}{2}}}\|\Lambda^{\varrho+\frac{\beta}{2}}\theta\|_{L^{2}}
\nonumber\\
&\leq& C(\|\nabla
u\|_{L^{2}}\|\theta\|_{B_{\infty,2}^{\varrho-\frac{\beta}{2}}}
+\|u\|_{L^{2}}\|\theta\|_{L^{2}})
\|\Lambda^{\varrho+\frac{\beta}{2}}\theta\|_{L^{2}}\qquad \Big(\varrho<\frac{\beta}{2}\Big)\nonumber\\
&\leq&
C(\|\omega\|_{L^{2}}\|\theta\|_{L^{\infty}}+\|u\|_{L^{2}}\|\theta\|_{L^{2}})
\|\Lambda^{\varrho+\frac{\beta}{2}}\theta\|_{L^{2}}
\nonumber\\
&\leq& C(\|G\|_{L^{2}}\|\theta\|_{L^{\infty}}+\|\mathcal
{R}_{\beta}\theta\|_{L^{2}}\|\theta\|_{L^{\infty}}+\|u\|_{L^{2}}\|\theta\|_{L^{2}})
\|\Lambda^{\varrho+\frac{\beta}{2}}\theta\|_{L^{2}}
\nonumber\\
&\leq&
\frac{1}{4}\|\Lambda^{\varrho+\frac{\beta}{2}}\theta\|_{L^{2}}^{2}+
C(1+\|G\|_{L^{2}}^{2}+\|\theta\|_{H^{\frac{\beta}{2}}}^{2}),
\end{eqnarray}
where we have used the following facts
$$\|f\|_{H^{1-\frac{\beta}{2}}}\approx\|f\|_{B_{2,2}^{1-\frac{\beta}{2}}}\,\,\,\mbox{and}\,\,\,
\|\mathcal {R}_{\beta}\theta\|_{L^{2}}\leq
C\|\theta\|_{H^{\frac{\beta}{2}}},\,\,\beta\geq\frac{2}{3}.$$
Inserting the above estimate in (\ref{AZ002}), we thus obtain
\begin{eqnarray}\label{Vor02}\frac{1}{2}\frac{d}{dt}\|\Lambda^{\varrho}\theta\|_{L^{2}}^{2}
+\frac{3}{4}\|\Lambda^{\varrho+\frac{\beta}{2}}
\theta\|_{L^{2}}^{2}\leq
 C(1+\|G\|_{L^{2}}^{2}+\|\theta\|_{H^{\frac{\beta}{2}}}^{2}).
\end{eqnarray}
In order to close the above inequality, we need to consider the
equation (\ref{t305}). To this end, we multiply the equation (\ref{t305}) by $G$
to obtain
\begin{eqnarray}\label{Vo308}\frac{1}{2}\frac{d}{dt}\|G\|_{L^{2}}^{2}
+\|\Lambda^{\frac{\alpha}{2}}G\|_{L^{2}}^{2}&=&\int_{\mathbb{R}^{2}}
\Lambda^{\alpha-\beta}\partial_{x_{1}}\theta \,\,G\,dx\ -
\int_{\mathbb{R}^{2}}
[\mathcal {R}_{\beta},\,u\cdot\nabla]\theta \,\,G\,dx\nonumber\\
&=& \int_{\mathbb{R}^{2}}
\Lambda^{\alpha-\beta}\partial_{x_{1}}\theta \,\,G\,dx\ -
\int_{\mathbb{R}^{2}} [\mathcal
{R}_{\beta},\,u_{G}\cdot\nabla]\theta
\,\,G\,dx\nonumber\\&&-\int_{\mathbb{R}^{2}} [\mathcal
{R}_{\beta},\,u_{\theta}\cdot\nabla]\theta \,\,G\,dx.
\end{eqnarray}
Bounding the first term at the R-H-S of (\ref{Vo308}) according to
the H$\rm \ddot{o}$lder inequality and the interpolation inequality,
we thus get
\begin{eqnarray}\label{Vo309}\int_{\mathbb{R}^{2}}
\Lambda^{\alpha-\beta}\partial_{x_{1}}\theta \,\,G\,dx &\leq&
C\|\Lambda^{1+\frac{\alpha}{2}-\beta}\theta\|_{L^{2}}
\|\Lambda^{\frac{\alpha}{2}}G\|_{L^{2}}\nonumber\\
&\leq&
C\|\theta\|_{L^{2}}^{1-\tau}\|\Lambda^{\varrho+\frac{\beta}{2}}
\theta\|_{L^{2}}^{\tau}
\|\Lambda^{\frac{\alpha}{2}}G\|_{L^{2}}\nonumber\\&&
\Big(\varrho>\frac{2
+\alpha-3\beta}{2}\Rightarrow \tau=\frac{2+\alpha-2\beta}{2\varrho+\beta}\in (0,\,1)\Big)\nonumber\\
&\leq&\frac{1}{8}\|\Lambda^{\frac{\alpha}{2}}G\|_{L^{2}}^{2}+
\frac{1}{8}\|\Lambda^{\varrho+\frac{\beta}{2}}\theta\|_{L^{2}}^{2}
+C\|\theta\|_{L^{2}}^{2}.
\end{eqnarray}
Next we appeal to the commutator estimate (\ref{t203}) and the
fractional type Gagliardo-Nirenberg inequality (\ref{t205}) to bound
the third term at the R-H-S of (\ref{Vo308})
\begin{eqnarray}\label{Vo31011}\int_{\mathbb{R}^{2}}
[\mathcal {R}_{\beta},\,u_{\theta}\cdot\nabla]\theta \,\,G\,dx
&\leq& \|[\mathcal
{R}_{\beta},\,u_{\theta}\cdot\nabla]\theta\|_{L^{2}}
\|G\|_{L^{2}}\nonumber\\
&\leq& C(\|\nabla u_{\theta}\|_{L^{4}}\|\theta\|_{B_{4,1}^{1-\beta}}
+\|u_{\theta}\|_{L^{\frac{2}{1-\beta}}}\|\theta\|_{L^{2}})\|G\|_{L^{2}}\nonumber\\
&\leq&
C\|\Lambda^{1-\beta}\theta\|_{L^{4}}\|\theta\|_{B_{4,1}^{1-\beta}}
\|G\|_{L^{2}}+\|\theta\|_{L^{2}}^{2}\|G\|_{L^{2}}\nonumber\\
&\leq& C\|\theta\|_{B_{4,1}^{1-\beta}}^{2} \|G\|_{L^{2}}+\|\theta_{0}\|_{L^{2}}^{2}\|G\|_{L^{2}},
\end{eqnarray}
where we have used the classical embedding
$B_{4,1}^{1-\beta}\hookrightarrow W^{1-\beta,4}.$ Thanks to
(\ref{t205}), we have that for $2-2\beta<s<\frac{3-2\beta}{2}$ with
$\frac{1}{2}<\beta<1$
\begin{align}\label{COD1}
\|\theta\|_{{B}_{4,1}^{1-\beta}(\mathbb{R}^{2})}\leq C\|\theta\|_{
{B}_{2,2}^{s}(\mathbb{R}^{2})}^{\lambda}
\|\theta\|_{L^{\infty}(\mathbb{R}^{2})}^{1-\lambda},\quad
\lambda=\frac{2\beta-1}{2-2s}\in (0,\,1).\end{align} According to
Sobolev interpolation, we can get that for any
$\frac{\beta}{2}<s<\frac{\beta}{2}+\varrho$
\begin{align}\label{COD2}
\|\theta\|_{ {B}_{2,2}^{s}(\mathbb{R}^{2})}\leq
\|\theta\|_{H^{\frac{\beta}{2}}(\mathbb{R}^{2})}^{1-l}
\|\Lambda^{\varrho+\frac{\beta}{2}}\theta\|_{L^{2}(\mathbb{R}^{2})}^{l},
\quad l=\frac{s-\frac{\beta}{2}}{\varrho}\in (0,\,1).
\end{align}
Inserting (\ref{COD2}) into (\ref{COD1}) and considering
(\ref{Vo31011}), we can conclude
\begin{eqnarray}\label{Vo310}\int_{\mathbb{R}^{2}}
[\mathcal {R}_{\beta},\,u_{\theta}\cdot\nabla]\theta \,\,G\,dx
&\leq& C\|\theta\|_{ {B}_{2,2}^{s}(\mathbb{R}^{2})}^{2\lambda}
\|\theta\|_{L^{\infty}(\mathbb{R}^{2})}^{2(1-\lambda)}
\|G\|_{L^{2}}\nonumber\\
&\leq&
C\|\theta\|_{H^{\frac{\beta}{2}}}^{2\lambda(1-l)}\|\Lambda^{\varrho+\frac{\beta}{2}}\theta\|_{L^{2}}
^{2\lambda l}
\|\theta_{0}\|_{L^{\infty}(\mathbb{R}^{2})}^{2(1-\lambda)}
\|G\|_{L^{2}}\nonumber\\
&\leq&
\frac{1}{8}\|\Lambda^{\varrho+\frac{\beta}{2}}\theta\|_{L^{2}}^{2}
+C\|\theta\|_{H^{\frac{\beta}{2}}}^{\frac{2\lambda(1-l)}{1-\lambda
l}}(1+\|G\|_{L^{2}}^{2}),
\end{eqnarray}
where in the last line we have used the following fact
$$ 0<s<\frac{\frac{\beta}{2}(2\beta-1)+\varrho}{2\beta-1+\varrho} \Rightarrow \lambda l\leq\frac{1}{2}\Leftrightarrow \frac{2\beta-1}{2-2s}\frac{s-\frac{\beta}{2}}{\varrho}\leq\frac{1}{2}\Rightarrow \frac{1}{1-\lambda l}\leq 2 .$$
Combing all the restrictions on $s$ yields
$$\max\Big\{0,\,\,2-2\beta,\,\,\frac{\beta}{2}\Big\}<s<\min\Big\{\frac{3-2\beta}{2},\,\,
\frac{\beta}{2}+\varrho,\,\,\frac{\frac{\beta}{2}(2\beta-1)+\varrho}{2\beta-1+\varrho}\Big\},$$
which would work as long as
$$\varrho>\frac{4-5\beta}{2},\quad \beta>\frac{1}{2}.$$
Now we focus on the second term at the R-H-S of (\ref{Vo308}). The
estimate (\ref{t206}) as well as energy estimate (\ref{t301}) leads
to
\begin{eqnarray}\label{Vo311}
\|\Lambda^{\gamma\beta}\theta\|_{L_{t}^{\frac{1}{\gamma}}L_{x}^{\frac{1}{\gamma}}}
\leq
C\|\Lambda^{\frac{\beta}{2}}\theta\|_{L_{t}^{2}L_{x}^{2}}^{2\gamma}
\|\theta\|_{L_{t}^{\infty}L_{x}^{\infty}}^{1-2\gamma}<\infty,\quad
0<\gamma<\frac{1}{2}.
\end{eqnarray}
The commutator estimate (\ref{t202}) with
$s_{1}=\gamma\beta<1-\beta,\,\,1-\beta-\gamma\beta<s_{2}<
\frac{\alpha}{2},\,\,p_{1}=\frac{1}{\gamma},\,\,
p_{2}=\frac{4}{2+2s_{2}-\alpha}$ and
$p_{3}=\frac{4}{2+\alpha-2s_{2}-4\gamma}$ allows us to show
\begin{eqnarray}\label{Vo312}&&\int_{\mathbb{R}^{2}}
[\mathcal {R}_{\beta},\,u_{G}\cdot\nabla]\theta \,\,G\,dx \nonumber\\ &\leq&
C\|\Lambda^{\gamma\beta}\theta\|_{L^{\frac{1}{\gamma}}}\|G\|_{W^{s_{2},\,p_{2}}}
\|G\|_{L^{p_{3}}}\nonumber\\
&\leq&C\|\Lambda^{\gamma\beta}\theta\|_{L^{\frac{1}{\gamma}}}
\|\Lambda^{\frac{\alpha}{2}}G\|_{L^{2}}
(\|G\|_{L^{2}}^{1-\mu}\|\Lambda^{\frac{\alpha}{2}}G\|_{L^{2}}^{\mu})\nonumber\\
&&
\Big(\frac{\alpha-4\gamma}{2}< s_{2}<
\alpha-(\alpha+2)\gamma\Rightarrow\mu=\frac{2s_{2}+4\gamma-\alpha}{\alpha}\in (0,\,1)\Big)\nonumber\\
&\leq&\frac{1}{8}\|\Lambda^{\frac{\alpha}{2}}G\|_{L^{2}}^{2}
+C\|\Lambda^{\gamma\beta}\theta\|_{L^{\frac{1}{\gamma}}}^{\frac{2}{1-\mu}}
\|G\|_{L^{2}}^{2}\nonumber\\
&\leq&\frac{1}{8}\|\Lambda^{\frac{\alpha}{2}}G\|_{L^{2}}^{2}
+C(1+\|\Lambda^{\gamma\beta}\theta\|_{L^{\frac{1}{\gamma}}}^{\frac{1}{\gamma}}
)\|G\|_{L^{2}}^{2},
\end{eqnarray}
where in the last line the following fact has been applied
$$\frac{2}{1-\mu}\leq \frac{1}{\gamma}\Rightarrow
\mu=\frac{2s_{2}+4\gamma-\alpha}{\alpha}\leq 1-2\gamma.$$ Putting
all the restrictions on $s_{2}$ together, we have
$$1-\beta-\gamma\beta<s_{2}<
\frac{\alpha}{2},\qquad \frac{\alpha-4\gamma}{2}< s_{2}<
\alpha-(\alpha+2)\gamma.$$ Consequently, the above $s_{2}$ would
work as long as
$$\max\Big\{0,\,\,\frac{2-2\beta-\alpha}{2\beta}\Big\}<\gamma<\min\Big\{\frac{1}{2},\,\,
\frac{1-\beta}{\beta},\,\,
\frac{\alpha+\beta-1}{2+\alpha-\beta}\Big\},$$ which leads to the
key assumption
$$\beta>\frac{4-\alpha^{2}}{4+3\alpha}.$$
It is worth noting that the fact $\frac{4-\alpha^{2}}{4+3\alpha}>1-\alpha$ and this is the only place
where we use the assumption $\beta>\frac{4-\alpha^{2}}{4+3\alpha}$.
Inserting the above aforementioned estimates (\ref{Vo309}),
(\ref{Vo310}) and (\ref{Vo312}) into (\ref{Vo308}) yields
\begin{eqnarray}\label{Vo313}\frac{1}{2}\frac{d}{dt}\|G\|_{L^{2}}^{2}
+\frac{3}{4}\|\Lambda^{\frac{\alpha}{2}}G\|_{L^{2}}^{2}&\leq&
\frac{1}{8}\|\Lambda^{\varrho+\frac{\beta}{2}}\theta\|_{L^{2}}^{2}
+C\|\theta\|_{L^{2}}^{2}+C(1+\|\theta\|_{H^{\frac{\beta}{2}}}^{2})
(1+\|G\|_{L^{2}}^{2})\nonumber\\&&
+C(1+\|\Lambda^{\gamma\beta}\theta\|_{L^{\frac{1}{\gamma}}}^{\frac{1}{\gamma}}
)\|G\|_{L^{2}}^{2}.
\end{eqnarray}
Summing up (\ref{Vo313}) and (\ref{Vor02}), we thereby obtain
$$\frac{d}{dt}(\|G\|_{L^{2}}^{2}+\|\Lambda^{\varrho}\theta\|_{L^{2}}^{2})
+\|\Lambda^{\frac{\alpha}{2}}G\|_{L^{2}}^{2}+\|\Lambda^{\varrho
+\frac{\beta}{2}}\theta\|_{L^{2}}^{2}\leq
C(1+\|\theta\|_{H^{\frac{\beta}{2}}}^{2}+\|\Lambda^{\gamma\beta}\theta\|_{L^{\frac{1}{\gamma}}}^{\frac{1}{\gamma}}
)(1+\|G\|_{L^{2}}^{2}),$$ which together with the classical Gronwall
inequality and (\ref{Vo311}) lead to
$$\sup_{0\leq t\leq T}(\|G(t)\|_{L^{2}}^{2}+\|\Lambda^{\varrho}\theta(t)\|_{L^{2}}^{2})
+\int_{0}^{T}{\big(\|\Lambda^{\frac{\alpha}{2}}G\|_{L^{2}}^{2}+\|\Lambda^{\varrho+
\frac{\beta}{2}}\theta\|_{L^{2}}^{2}\big)(\tau)\,d\tau}\leq
C(T,\,u_{0},\,\theta_{0}).$$ This completes the proof of Lemma
\ref{AZL302}.
\end{proof}

\vskip .3in
\subsection{\textbf{Case 2}}
In this case, we consider the vorticity $\omega$ instead of the
combined quantity $G$. Now we derive the following estimates
concerning vorticity $\omega$ and the temperature $\theta$, which
can be stated as follows.
\begin{lemma}\label{TAZL302}
Assume $(u_{0},\,\theta_{0})$ satisfies the assumptions stated in
Theorem \ref{Th1}. Let $(u, \theta)$ be the corresponding solution
of the system (\ref{Bouss}). If $\beta>\frac{2-\alpha}{2}$, then the
vorticity $\omega$ and the temperature $\theta$ admit the following
bound for any $\frac{2-\alpha-\beta}{2}<\delta<\frac{\beta}{2}$ and
$t\in[0, T]$
\begin{eqnarray}\label{TAZ001}
\|\omega(t)\|_{L^{2}}^{2}+\|\Lambda^{\delta}\theta(t)\|_{L^{2}}^{2}
+\int_{0}^{t}{\big(\|\Lambda^{\frac{\alpha}{2}}\omega\|_{L^{2}}^{2}+
\|\Lambda^{\delta+
\frac{\beta}{2}}\theta\|_{L^{2}}^{2}\big)(\tau)\,d\tau}\leq
C(T,\,u_{0},\,\theta_{0}),
\end{eqnarray}
where $C(T,\,u_{0},\,\theta_{0})$ is a constant depending on $T$ and
the initial data.
\end{lemma}
\begin{rem}\rm
Similarly, by energy estimate (\ref{t301}) and the classical
interpolation, we find that (\ref{TAZ001}) is true for any $0\leq
\delta<\frac{\beta}{2}$.
\end{rem}

\begin{proof}[\textbf{Proof of Lemma \ref{TAZL302}}]
With the same argument used in obtaining (\ref{Fi01}), we find that
\begin{eqnarray}\label{TAZ002}\frac{1}{2}\frac{d}{dt}\|\Lambda^{\delta}\theta\|_{L^{2}}^{2}
+\|\Lambda^{\delta+\frac{\beta}{2}}\theta\|_{L^{2}}^{2}
&\leq& C\|\Lambda^{1-\frac{\beta}{2}}[\Lambda^{\delta}, u]\theta\|_{L^{2}}
\|\Lambda^{\delta+\frac{\beta}{2}}\theta\|_{L^{2}}\nonumber\\
&\leq& C\|[\Lambda^{\delta}, u ]\theta\|_{H^{1-\frac{\beta}{2}}}\|\Lambda^{\delta+\frac{\beta}{2}}\theta\|_{L^{2}}\nonumber\\
&\leq& C\|[\Lambda^{\delta}, u
]\theta\|_{B_{2,2}^{1-\frac{\beta}{2}}}\|\Lambda^{\delta
+\frac{\beta}{2}}\theta\|_{L^{2}}
\nonumber\\
&\leq& C(\|\nabla
u\|_{L^{2}}\|\theta\|_{B_{\infty,2}^{\delta-\frac{\beta}{2}}}
+\|u\|_{L^{2}}\|\theta\|_{L^{2}})
\|\Lambda^{\delta+\frac{\beta}{2}}\theta\|_{L^{2}}\quad
\Big(\delta<\frac{\beta}{2}\Big)\nonumber\\
&\leq&
C(\|\omega\|_{L^{2}}\|\theta\|_{L^{\infty}}+\|u\|_{L^{2}}\|\theta\|_{L^{2}})
\|\Lambda^{\delta+\frac{\beta}{2}}\theta\|_{L^{2}}
\nonumber\\
&\leq&
\frac{1}{4}\|\Lambda^{\delta+\frac{\beta}{2}}\theta\|_{L^{2}}^{2}+
C\|\theta_{0}\|_{L^{\infty}}^{2}\|\omega\|_{L^{2}}^{2}+C\|u\|_{L^{2}}^{2}
\|\theta\|_{L^{2}}^{2}.\nonumber
\end{eqnarray}
Substituting the above estimate into (\ref{TAZ002}), we thus obtain
\begin{eqnarray}\label{TVor02}\frac{1}{2}\frac{d}{dt}\|\Lambda^{\delta}\theta\|_{L^{2}}^{2}
+\frac{3}{4}\|\Lambda^{\delta+\frac{\beta}{2}}\theta\|_{L^{2}}^{2}\leq
C\|\theta_{0}\|_{L^{\infty}}^{2}\|\omega\|_{L^{2}}^{2}.
\end{eqnarray}
In order to obtain the global $H^{1}$ bound of the velocity $u$, we resort to
the vorticity $w$ equation (\ref{t303})
\begin{eqnarray}\label{TTes3}\partial_{t}w+(u\cdot\nabla)w+\Lambda^{\alpha} w=\partial_{x_{1}}\theta.\end{eqnarray}
Testing it by $\omega$ yields
\begin{eqnarray}\label{TVor01}\frac{1}{2}\frac{d}{dt}\|\omega\|_{L^{2}}^{2}
+\|\Lambda^{\frac{\alpha}{2}}\omega\|_{L^{2}}^{2}&=&\int_{\mathbb{R}^{2}}
\partial_{x_{1}}\theta \,\,\omega\,dx\nonumber\\
&\leq& C\|\Lambda^{1-\frac{\alpha}{2}}\theta\|_{L^{2}}
\|\Lambda^{\frac{\alpha}{2}}\omega\|_{L^{2}}\nonumber\\
&\leq&
C\|\Lambda^{\delta}\theta\|_{L^{2}}^{1-\tau}\|\Lambda^{\delta+\frac{\beta}{2}}
\theta\|_{L^{2}}^{\tau}
\|\Lambda^{\frac{\alpha}{2}}\omega\|_{L^{2}}\nonumber\\
&\leq&\frac{1}{2}\|\Lambda^{\frac{\alpha}{2}}\omega\|_{L^{2}}^{2}+
\frac{1}{4}\|\Lambda^{\delta+\frac{\beta}{2}}\theta\|_{L^{2}}^{2}
+C\|\Lambda^{\delta}\theta\|_{L^{2}}^{2},
\end{eqnarray}
where we have applied the following Sobolev interpolation
$$\|\Lambda^{1-\frac{\alpha}{2}}\theta\|_{L^{2}}\leq C\|\Lambda^{\delta}
\theta\|_{L^{2}}^{1-\tau}\|\Lambda^{\delta+\frac{\beta}{2}}
\theta\|_{L^{2}}^{\tau},\quad \tau=\frac{2-\alpha-2\delta}{\beta}\in
(0,\,1).$$ Note the fact
$$\frac{2-\alpha-\beta}{2}<\delta<\frac{2-\alpha}{2}\Rightarrow 0<\tau<1.$$
Putting all the restrictions on $\delta$ together, we have
$$\frac{2-\alpha-\beta}{2}<\delta<\min\Big\{\frac{2-\alpha}{2},\,\,\frac{\beta}{2}
\Big\}=\frac{\beta}{2}.$$ Thus, this is the only place in the proof
where we use the assumption of the
theorem, namely $\beta>\frac{2-\alpha}{2}$.\\
Summing up (\ref{TVor01}) and (\ref{TVor02}), we thereby obtain
$$\frac{d}{dt}(\|\omega\|_{L^{2}}^{2}+\|\Lambda^{\delta}\theta\|_{L^{2}}^{2})
+\|\Lambda^{\frac{\alpha}{2}}\omega\|_{L^{2}}^{2}+\|\Lambda^{\delta
+\frac{\beta}{2}}\theta\|_{L^{2}}^{2}\leq
C(\|\omega\|_{L^{2}}^{2}+\|\Lambda^{\delta}\theta\|_{L^{2}}^{2})+
C\|u\|_{L^{2}}^{2} \|\theta\|_{L^{2}}^{2},$$ which together with the
classical Gronwall inequality leads to
$$\sup_{0\leq t\leq T}(\|\omega\|_{L^{2}}^{2}+
\|\Lambda^{\delta}\theta\|_{L^{2}}^{2})(t)
+\int_{0}^{T}{\big(\|\Lambda^{\frac{\alpha}{2}}
\omega\|_{L^{2}}^{2}+\|\Lambda^{\delta+
\frac{\beta}{2}}\theta\|_{L^{2}}^{2}\big)(\tau)\,d\tau}<\infty.$$
This completes the proof of Lemma \ref{TAZL302}.
\end{proof}

\vskip .3in
Both in \textbf{Case 1} and \textbf{Case 2}, we can establish the
following global {\it a priori} bound $\|u(t)\|_{L^{r}}$ for any
$2\leq r<\infty$ and $0\leq t\leq T$.
\begin{lemma}\label{L304}
Assume $\beta$ satisfies the assumptions stated in Lemmas \ref{AZL302} and \ref{TAZL302}, then the velocity field $u$ obeys
the following key global {\it a priori} bound for any $2\leq
r<\infty$ and $0\leq t\leq T$
\begin{eqnarray}\label{tNew03}
\sup_{0\leq t\leq T}\|u(t)\|_{L^{r}}\leq
C(r,\,T,\,u_{0},\,\theta_{0}),
\end{eqnarray}
where $C(r,\,T,\,u_{0},\,\theta_{0})$ is a constant depending on
$r$, $T$ and the initial data.
\end{lemma}
\begin{proof}[\textbf{Proof of Lemma \ref{L304}}]
Let us notice that in Case 1, we have $\beta>\frac{2}{3}$. As a
result, we can select $\varrho$ satisfying
$1-\beta<\varrho<\frac{\beta}{2}$ such that
$$\|\mathcal{R}_{\beta}\theta\|_{L^{2}}\leq \|\Lambda^{1-\beta}\theta
\|_{L^{2}}\leq \|\theta \|_{H^{\varrho}}<\infty.$$ Recalling
$G=\omega-\mathcal {R}_{\beta}\theta$ and the estimate
(\ref{AZ001}), we get
$$\|\omega\|_{L^{2}}\leq \|G\|_{L^{2}}+\|\mathcal{R}_{\beta}
\theta\|_{L^{2}}<\infty,$$ which together with (\ref{TAZ001})
implies that we have both in \textbf{Case 1} and \textbf{Case 2}
$$\sup_{0\leq t\leq T}\|\omega(t)\|_{L^{2}}<\infty.$$
By the Sobolev interpolation inequality
\begin{eqnarray}\sup_{0\leq t\leq T}\|u(t)\|_{L^{r}}&\leq & C(r)\|u\|_{L^{2}}^{\frac{2}{r}}\|\nabla u\|_{L^{2}}^{1-\frac{2}{r}}
\nonumber\\&\leq&
C(r)\|u\|_{L^{2}}^{\frac{2}{r}}\|\omega\|_{L^{2}}^{1-\frac{2}{r}}
\nonumber\\&\leq& C(r,\,T,\,u_{0},\,\theta_{0}).
\end{eqnarray}
Consequently, this immediately completes the proof of Lemma
\ref{L304}.
\end{proof}

With the help of the above estimate (\ref{tNew03}), we are able to
the next lemma, which is concerned with the global {\it a priori} bounda
$\|\nabla\theta \|_{L^{\infty}}$
 as well as $\|\omega\|_{L^{\infty}}$.
\begin{lemma}\label{L305}
Assume $(u_{0},\,\theta_{0})$ satisfies the assumptions stated in
Theorem \ref{Th1}. Assume $\beta$ satisfies the assumptions stated in Lemmas \ref{AZL302} and \ref{TAZL302}, then the
temperature $\theta$ and the vorticity $\omega$ admit the following key global {\it a priori} bound
\begin{eqnarray}\label{tNew007}
\sup_{0\leq t\leq T}\|\nabla\theta (t)\|_{L^{\infty}}\leq
C(T,\,u_{0},\,\theta_{0}),
\end{eqnarray}
\begin{eqnarray}\label{Vorticity}\sup_{0\leq t\leq T}\|\omega(t)\|_{L^{\infty}}\leq
C(T,\,u_{0},\,\theta_{0}),
\end{eqnarray}
where $C(T,\,u_{0},\,\theta_{0})$ is a constant depending on $T$ and the initial data.
\end{lemma}

\begin{proof}[\textbf{Proof of Lemma \ref{L305}}]
The idea of the proof is based on the argument
of nonlinear lower bounds for the fractional Laplacian established in \cite{CV}. For convenience the reader, we present the details as follows.
By the elementary calculations, it is not hard to check that
$$\beta^{\ast}\geq\frac{1}{1+\alpha}.$$
Therefore, this fact further implies
\begin{eqnarray}\beta>\frac{1}{1+\alpha}.\nonumber\end{eqnarray}
We start with the following pointwise bound
\begin{eqnarray}\label{tNew008}
\nabla f(x)\cdot \Lambda^{\alpha}\nabla f(x)\geq \frac{1}{2}\Lambda^{\alpha}(|\nabla f(x)|^{2})+\frac{|\nabla f(x)|^{2+\frac{\alpha p}{p+2}}}{c\|f\|_{L^{p}}^{\frac{\alpha p}{p+2}}},
\end{eqnarray}
which can be proved by combining the proofs of Theorems 2.2 and 2.5 of \cite{CV}.\\
Applying $\nabla$ to the temperature equation of (\ref{Bouss}) and multiplying the resulting equation by $\nabla \theta$ lead to
\begin{eqnarray}\label{tNew009}
\frac{1}{2}(\partial_{t}+u \cdot \nabla)|\nabla\theta|^{2} +\nabla\theta\cdot\Lambda^{\beta}\nabla\theta=-\nabla u:\nabla\theta\cdot\nabla\theta.
\end{eqnarray}
Thus, by making use of (\ref{tNew008}) with $p=\infty$, we immediately arrive at
\begin{eqnarray}\label{tNew010}
\frac{1}{2}(\partial_{t}+u \cdot \nabla+\Lambda^{\beta})|\nabla\theta|^{2}+c_{1}\frac{
|\nabla\theta(x)|^{2+\beta}}{\|\theta_{0}\|_{L^{\infty}}^{\beta}}\leq-\nabla u:\nabla\theta\cdot\nabla\theta.
\end{eqnarray}
Suppose that $|\nabla\theta(x,t)|$ achieves the maximum at the point $\widetilde{x}=\widetilde{x}(t)$, then we get
\begin{eqnarray}\label{tNew011}
\partial_{t}|\nabla\theta(\widetilde{x},t)|^{2}+c_{1}\frac{
\Phi(t)^{2+\beta}}{\|\theta_{0}\|_{L^{\infty}}^{\beta}}\leq\Phi(t)^{2}\|\nabla u\|_{L^{\infty}},
\end{eqnarray}
where
$$\Phi(t)=\|\nabla \theta(., t)\|_{L^{\infty}}.$$
Similarly, let us assume that $|\omega(x,t)|$ achieves the maximum at the point $\widehat{x}=\widehat{x}(t)$ and denote
$$\Omega(t)=\|\omega(., t)\|_{L^{\infty}}.$$
Recalling the vorticity equation
\begin{eqnarray}\partial_{t}w+(u\cdot\nabla)w+\Lambda^{\alpha} w=\partial_{x_{1}}\theta,\nonumber\end{eqnarray}
and adapting the same argument used above, we can conclude that
\begin{eqnarray}\label{tNew012}
\partial_{t}|\omega(\widehat{x},t)|^{2}+c_{2}\frac{
\Omega(t)^{2+\frac{\alpha r}{2+r}}}{\|u\|_{L^{r}}^{\frac{\alpha r}{2+r}}}\leq\Phi(t) \Omega(t),
\end{eqnarray}
where the number $r\in (2,\,\infty)$ will be fixed hereafter.\\
To bound $\| \nabla u\|_{L^{\infty}}$, we need the following logarithmic inequality which was established in (\cite{CV})
\begin{eqnarray}\label{tNew013}
\|\nabla u(.,t)\|_{L^{\infty}}&\leq& C_{0}+C_{0}\Omega(t)\nonumber\\&&+C_{0}\Omega(t)\log_{+}
\Big(1+\int_{0}^{t}{\big(1+K(\tau)+\Omega(\tau)+\Phi(\tau)\big)^{\Gamma}\,d\tau}\Big)
\end{eqnarray}
where $C_{0}>0$ is a constant depending on initial data, $K(\tau)$ is a bounded function on the interval $[0,\,T]$ and $\Gamma=\Gamma(\alpha,\beta)$.\\
Therefore, it follows from (\ref{tNew011}) and (\ref{tNew012}) that
\begin{eqnarray}\label{tNew014}
 \partial_{t}|\nabla\theta(\widetilde{x},t)|^{2}+C_{1}
\Phi(t)^{2+\beta} &\leq& C_{0}\Phi(t)^{2}\Big\{1+ \Omega(t) + \Omega(t)\log_{+}
\Big(1\nonumber\\&&+\int_{0}^{t}{\big(1+K(\tau)+\Omega(\tau)+\Phi(\tau)\big)^{\Gamma}
\,d\tau}\Big)\Big\},
\end{eqnarray}
\begin{eqnarray}\label{tNew015}
\partial_{t}|\omega(\widehat{x},t)|^{2}+C_{2}
\Omega(t)^{2+\frac{\alpha r}{2+r}} \leq\Phi(t) \Omega(t),
\end{eqnarray}
where $\Phi(t)=|\nabla \theta(\widetilde{x}, t)|$ and $\Omega(t)= |\omega(\widehat{x}, t)|$, and the constants $C_{0}$, $C_{1}$, $C_{2}$ depend on the initial data, $\alpha$, $\beta$ and $\|u(T)\|_{L^{r}}$.\\
Suppose that $M>0$ is large enough to be fixed hereafter. Assuming the solutions blow up at time $T$, thus $\lim_{t\rightarrow T}\Phi(t)=\infty$. Then we can select $T_{0}\in (0,\,T)$ as the first time such that $\Phi(T_{0})=M\geq 4\Phi(0)$. Now one may deduce from (\ref{tNew015}) that for any $t\in [0,\,T_{0}]$
$$\Omega(t)\leq\max\Big\{\Omega(0),\,\,\, \Big(\frac{M}{C_{2}}\Big)^{\frac{2+r}{2+(1+\alpha)r}}\Big\}
=\Big(\frac{M}{C_{2}}\Big)^{\frac{2+r}{2+(1+\alpha)r}}:=\widetilde{M},$$
as long as $M$ is large enough in terms of $\Omega(0)$, $\alpha$, $r$ and $C_{2}$.
Let us give details about how to get the above estimate. Actually, if $\Omega(t)\geq \widetilde{M}$, then
\begin{eqnarray}C_{2}
\Omega(t)^{2+\frac{\alpha r}{2+r}}-\Phi(t) \Omega(t)&\geq& C_{2}
\Omega(t)^{2+\frac{\alpha r}{2+r}}-M\Omega(t)\nonumber\\&\geq&
(C_{2}
\Omega(t)^{1+\frac{\alpha r}{2+r}}-M)\Omega(t)\nonumber\\&\geq&
(C_{2}
\widetilde{M}^{1+\frac{\alpha r}{2+r}}-M)\Omega(t)\nonumber\\
&=&0.\nonumber
\end{eqnarray}
Thus it follows from (\ref{tNew015}) that
$\partial_{t}|\omega(\widehat{x},t)|^{2}\leq0$. This implies that $\Omega(t)$
cannot exceed the value $\widetilde{M}$.
Hence, the following inequality is an easy consequence of (\ref{tNew014})
\begin{eqnarray}\label{tNew016}
\partial_{t}|\nabla\theta(\widetilde{x},t)|^{2}+C_{1}
\Phi(t)^{2+\beta}\leq C_{0}\Phi(t)^{2}\Big\{1+ \widetilde{M}+ \widetilde{M}\log_{+}
\Big(1+\big(1+K(T)+\widetilde{M}+M\big)^{\Gamma}
\Big)\Big\}.\nonumber
\end{eqnarray}
Repeating the same argument as above, we obtain
\begin{eqnarray}\label{tNew017}
\Phi(t)^{\beta}\leq \max\Big\{\Phi(0)^{\beta},\,\,\frac{C_{0}}{C_{1}}\Big(1+ \widetilde{M}+ \widetilde{M}\log_{+}
\Big(1+\big(1+K(T)+\widetilde{M}+M\big)^{\Gamma}
\Big)\Big)\Big\},\nonumber
\end{eqnarray}
for any $t\in [0,\,T_{0}]$.\\
Now notice that $\widetilde{M}\approx M^{\frac{2+r}{2+(1+\alpha)r}}$, we select $M$ large enough such that
$$\frac{C_{0}}{C_{1}}\Big(1+ \widetilde{M}+ \widetilde{M}\log_{+}
\Big(1+\big(1+K(T)+\widetilde{M}+M\big)^{\Gamma}
\Big)\leq \Big(\frac{M}{4}\Big)^{\beta},$$
which is equivalent to
\begin{eqnarray}\label{KeyCon1}1+M^{\frac{2+r}{2+(1+\alpha)r}} (1+\log_{+}M)\leq \frac{M^{\beta}}{C}.
\end{eqnarray}
Thanks to the fact $\beta>\frac{1}{1+\alpha}$, it is sufficient to choose $r$ as
$$r_{0}<r<\infty,\qquad r_{0}=\max\Big\{\frac{2(1-\beta)}{(1+\alpha)\beta-1},\,\,2\Big\},$$
then the above inequality (\ref{KeyCon1}) can be guaranteed due to the following fact
$$r_{0}<r\Rightarrow \frac{2+r}{2+(1+\alpha)r}<\beta .$$
Hence, it is not difficult to verify that $\Phi(T_{0})\leq\frac{M}{4}$, which contradict the definition of $T_{0}$. We thus get the fact that $\Phi(t)$ never blows up as $t\rightarrow T$ when $T<\infty$. As a direct consequence of above fact, we infer that
$$
\sup_{0\leq t\leq T}\|\nabla\theta (t)\|_{L^{\infty}}\leq
C(T,\,u_{0},\,\theta_{0})<\infty.
$$
As a consequence of the above estimate, it follows from the
vorticity equation (\ref{TTes3}) that for any $0\leq t\leq T$
$$\|\omega(t)\|_{L^{\infty}}\leq \|\omega_{0}\|_{L^{\infty}}
+\int_{0}^{t}{\|\nabla \theta(\tau)\|_{L^{\infty}}\,d \tau}\leq
C(T,\,u_{0},\,\theta_{0})<\infty.$$
This concludes the proof of Lemma \ref{L305}.
\end{proof}

\vskip .4in
\section{The proof of Theorem \ref{Th1}}\setcounter{equation}{0}
In this section we give the proof of Theorem \ref{Th1}. With the estimates (\ref{tNew007}) and (\ref{Vorticity}) at hand, the proof can be performed as the classical approach.
\begin{proof}[\textbf{Proof of Theorem \ref{Th1}}]
To begin with, we smooth the initial data to consider the following approximate system
\begin{equation}\label{APBouss}
\left\{\aligned
&\partial_{t}u^{(N)}+(u^{(N)} \cdot \nabla) u^{(N)}+\Lambda^{\alpha}u^{(N)}+\nabla p^{(N)}=\theta^{(N)} e_{2},\,\,\,\,\,\,\,\,x\in \mathbb{R}^{2},\,\,t>0, \\
&\partial_{t}\theta^{(N)}+(u^{(N)} \cdot \nabla) \theta^{(N)}+\Lambda^{\beta}\theta^{(N)}=0, \, \,\,\,\,\,\,\,\,\quad\quad\qquad\qquad x\in \mathbb{R}^{2},\,\,t>0,\\
&\nabla\cdot u^{(N)}=0, \,\,\,\,\,\,\,\qquad \qquad\qquad\qquad\qquad \qquad \qquad \qquad \quad x\in \mathbb{R}^{2},\,\,t>0,\\
&u^{(N)}(x, 0)=S_{N}u_{0}(x),  \quad \theta^{(N)}(x,0)=S_{N}\theta_{0}(x),\,\,\,\,\,\qquad\quad x\in \mathbb{R}^{2},
\endaligned\right.
\end{equation}
where $S_{N}$ is the low-frequency cut-off operator (see Appendix for it definition).

Now we apply
$(I+\Lambda)^{s}$ to
system (\ref{APBouss}) and multiply the resulting equations by
$(I+\Lambda)^{s}u^{(N)}$ and $(I+\Lambda)^{s}\theta^{(N)}$ respectively, add them up to to conclude that
\begin{eqnarray}\label{t316}
&&\frac{d}{dt}(\|u^{(N)}(t)\|_{H^{s}}^{2}+\|\theta^{(N)}(t)\|_{H^{s}}^{2})
+\|u^{(N)}\|_{H^{s+\frac{\alpha}{2}}}^{2}+\|\theta^{(N)}\|_{H^{s+\frac{\beta}{2}}}^{2}\nonumber\\
&\leq& C(1+\| \nabla u^{(N)}\|_{L^{\infty}}+\|\nabla \theta^{(N)}\|_{L^{\infty}})(\|{u}^{(N)}\|_{H^{s}}^{2}+\|\theta^{(N)}\|_{H^{s}}^{2})
\nonumber\\
&\leq&C(1+\|{u}^{(N)}\|_{H^{s}}+\|\theta^{(N)}\|_{H^{s}})
(\|{u}^{(N)}\|_{H^{s}}^{2}+\|\theta^{(N)}\|_{H^{s}}^{2}),
\end{eqnarray}
where we have used the embedding $H^{s}(\mathbb{R}^{2}))\hookrightarrow W^{1,\infty}(\mathbb{R}^{2}))$ for any $s>2$.\\
Therefore, there exists a time $$T^{\ast}:=C^{\ast}\Big(1+(\|u_{0}\|_{{H}^{s}}^{2}+\|\theta_{0}\|_{H^{s}}^{2})
\Big)^{-\frac{1}{2}}$$ for some absolute constant $C^{\ast}>0$ such that
$$u^{(N)}\in L^{\infty}([0, T^{\ast}); H^{s}(\mathbb{R}^{2}))\cap L^{2}([0, T^{\ast}); H^{s+\frac{\alpha}{2}}(\mathbb{R}^{2})),$$
$$\theta^{(N)}\in L^{\infty}([0, T^{\ast}); H^{s}(\mathbb{R}^{2}))\cap L^{2}([0, T^{\ast});
H^{s+\frac{\beta}{2}}(\mathbb{R}^{2})).$$
 Note that
$$\partial_{t}u^{(N)}=-\mathcal {P}\mathcal
((u^{(N)} \cdot \nabla) u^{(N)})-\Lambda^{\alpha}u^{(N)}+\mathcal {P}\mathcal
\theta^{(N)} e_{2},$$
$$\partial_{t}\theta^{(N)}=-(u^{(N)} \cdot \nabla) \theta^{(N)}-\Lambda^{\beta}\theta^{(N)},$$
where $\mathcal {P}$ denote the Leray projection onto
divergence-free vector fields.\\ Thus, it is not
hard to see that
$$\partial_{t}u^{(N)},\,\,\partial_{t}\theta^{(N)}\in L_{t}^{\infty}
([0, T^{\ast});\,H_{x}^{s-1}(\mathbb{R}^{2})).$$ Consequently, we assume that
$$\partial_{t}u^{(N)},\,\,\partial_{t}\theta^{(N)}\in L_{Loc}^{4}([0, T^{\ast});\,H_{x}^{s-1}(\mathbb{R}^{2})).$$
Since the embedding $H^{s}\hookrightarrow H^{s-1}$ is locally compact, the well-known Aubin-Lions argument and Cantor's diagonal
process, we conclude that there exists a solution satisfying
$$u \in L^{\infty}([0, T^{\ast}); H^{s}(\mathbb{R}^{2}))\cap L^{2}([0, T^{\ast}); H^{s+\frac{\alpha}{2}}(\mathbb{R}^{2})),$$
$$\theta \in L^{\infty}([0, T^{\ast}); H^{s}(\mathbb{R}^{2}))\cap L^{2}([0, T^{\ast});
H^{s+\frac{\beta}{2}}(\mathbb{R}^{2})).$$
The continuity of $u$ and $\theta$ in time, namely $u,\,\theta\in C([0, T^{\ast}); H^{s}(\mathbb{R}^{2}))$ can be obtained by a standard approach.
It suffices to consider $u\in C([0, T^{\ast}); H^{s}(\mathbb{R}^{2}))$ as the
same fashion can be applied to $\theta$ to obtain the desired result.\\
By the equivalent norm, it yields
\begin{align}\label{t2.03}
\|u(t_{1})-u(t_{2})\|_{H^{s}}=\Big\{(\sum_{j<N}+ \sum_{j\geq N})
(2^{js}\|\Delta_{j}u(t_{1})-\Delta_{j}u(t_{2})\|_{L^{2}})^{2}
\Big\}^{\frac{1}{2}},
\end{align}
 where $\Delta_{j}$ is the
non-homogeneous Littlewood-Paley operator (see Appendix for its definition).
Let $\varepsilon>0$ be arbitrarily small. Due to $u\in
L^{\infty}([0, T^{\ast}); H^{s}(\mathbb{R}^{2}))$, there exists a integer
$N>0$ such that
 \begin{align}\label{t2.04}
 \Big\{\sum_{j\geq N}
 (2^{js}\|\Delta_{j}u(t_{1})-\Delta_{j}u(t_{2})\|_{L^{2}})^{2}
 \Big\}^{\frac{1}{2}}<\frac{\varepsilon}{2}.
 \end{align}
Recalling the system $(\ref{Bouss})_{1}$, we obtain for $0\leq t_{1}<t_{2}<T^{\ast}$ that
\begin{eqnarray}\label{t2.05}
\Delta_{j}u(t_{1})-\Delta_{j}u(t_{2})&=&\int_{t_{1}}^{t_{2}}{\frac{d}{d\tau}
\Delta_{j}u(\tau)\,d\tau}\nonumber\\
&=&\int_{t_{1}}^{t_{2}}{ \Delta_{j}\mathcal {P}[\theta e_{2}-(u\cdot\nabla) u-\Lambda^{\alpha}u](\tau)\,d\tau}.
\end{eqnarray}
Therefore, we can get
\begin{eqnarray}\label{t2.06}
&&\sum_{j<N}
 2^{2js}\|\Delta_{j}u(t_{1})-\Delta_{j}u(t_{2})\|_{L^{2}}^{2}\nonumber\\
 &=&\sum_{j<N}
 2^{2js}\Big(\Big\|\int_{t_{1}}^{t_{2}}{ \Delta_{j}\mathcal {P}[\theta e_{2}-(u\cdot\nabla) u-\Lambda^{\alpha}u](\tau)\,d\tau}\Big\|_{L^{2}}\Big)^{2}\nonumber\\
&\leq&\sum_{j<N}
 2^{2js}\Big(\int_{t_{1}}^{t_{2}}{ [\|\Delta_{j}\theta\|_{L^{2}}+ \|\Delta_{j}(u\cdot\nabla u)\|_{L^{2}}+
\|\|\Delta_{j}\Lambda^{\alpha}u\|_{L^{2}}](\tau)\,d\tau}\Big)^{2} \nonumber\\
&=&\sum_{j<N}
 2^{2j}\Big(\int_{t_{1}}^{t_{2}}{ [2^{j(s-1)}\|\Delta_{j}\theta\|_{L^{2}}+ 2^{j(s-1)}\|\Delta_{j}(u\cdot\nabla u\|_{L^{2}}+
2^{j(s-1+\alpha)}\|\Delta_{j}u\|_{L^{2}}](\tau)\,d\tau}\Big)^{2}\nonumber\\
&\leq&C\sum_{j<N}
 2^{2j}\Big(\|\theta\|_{H^{s-1}}^{2}|t_{1}-t_{2}|+\|(u\cdot\nabla)
u\|_{H^{s-1}}^{2}|t_{1}-t_{2}|+\|u\|_{H^{s-1+\alpha}}^{2}|t_{1}-t_{2}|\Big)\nonumber\\
&\leq&C\sum_{j<N}
 2^{2j}|t_{1}-t_{2}|\Big(\|\theta\|_{H^{s-1}}^{2}+
\|u\|_{L^{\infty}}^{2}\|\nabla u\|_{H^{s-1}}^{2}+\|\nabla u\|_{L^{\infty}}^{2}\| u\|_{H^{s-1}}^{2}
+\|u\|_{H^{s}}^{2}\Big)\nonumber\\
&\leq&C
 2^{2N}|t_{1}-t_{2}|\Big(\|u\|_{H^{s}}^{4}+\| u\|_{H^{s}}^{2}
+\|\theta\|_{H^{s}}^{2}\Big),
\end{eqnarray}
where the Sobolev imbeddings $H^{s}(\mathbb{R}^{2})\hookrightarrow
H^{s-1}(\mathbb{R}^{2})$ and $H^{s-1}(\mathbb{R}^{2})\hookrightarrow
L^{\infty}(\mathbb{R}^{2})$  with $s>2$ are used several
times in the last inequality.\\
Thus, the following holds true
 \begin{align}\label{t2.07}
 \Big\{\sum_{j< N}
 (2^{js}\|\Delta_{j}u(t_{1})-\Delta_{j}u(t_{2})\|_{L^{2}})^{2}
 \Big\}^{\frac{1}{2}}<\frac{\varepsilon}{2}
 \end{align}
provided $|t_{1}-t_{2}|$ small enough.\\ Combining (\ref{t2.04}) with
(\ref{t2.07}) implies  $u\in C([0, T^{\ast}); H^{s}(\mathbb{R}^{2}))$.
Moreover, the uniqueness is clear since the velocity and the
temperature are both in Lipschitz spaces.

Now, it remains for us to show that the local smooth solutions may be extended to all positive time. It suffices to state
that under the assumption of the theorem and any given $T>0$, we have
$$\sup_{0\leq t\leq T}(\|{u}(t)\|_{H^{s}}^{2}+\|\theta(t)\|_{H^{s}}^{2})
+\int_{0}^{T}{\big(\|u(t)\|_{H^{s+\frac{\alpha}{2}}}^{2}+
\|\theta(t)\|_{H^{s+\frac{\beta}{2}}}^{2}\big)\,dt}\leq
C(T,\,u_{0},\,\theta_{0}).$$
In consequence, the energy estimate (\ref{t316}) ensures that
\begin{eqnarray}
&&\frac{d}{dt}(\|u(t)\|_{H^{s}}^{2}+\|\theta(t)\|_{H^{s}}^{2})
+\|u\|_{H^{s+\frac{\alpha}{2}}}^{2}+\|\theta\|_{H^{s+\frac{\beta}{2}}}^{2}\nonumber\\
&\leq& C(1+\| \nabla u\|_{L^{\infty}}+\|\nabla \theta\|_{L^{\infty}})(\|{u}\|_{H^{s}}^{2}+\|\theta\|_{H^{s}}^{2}).\nonumber
\end{eqnarray}
To obtain the global existence of smooth solutions, the standard procedure is to bound the term $\| \nabla u\|_{L^{\infty}}$ with $\|\omega\|_{L^{\infty}}$ and a Sobolev extrapolation inequality with logarithmic correction (see e.g., \cite{BKM,BG})
$$\|\nabla u\|_{L^{\infty}(\mathbb{R}^{2})} \leq C\Big(1+\|u\|_{L^{2}(\mathbb{R}^{2})}+
\|\omega\|_{L^{\infty}(\mathbb{R}^{2})} \ln(e+\|
u\|_{H^{s}(\mathbb{R}^{2})})\Big),\quad s>2.$$
Consequently, it enables us to get
\begin{eqnarray}
&&\frac{d}{dt}(\|u(t)\|_{H^{s}}^{2}+\|\theta(t)\|_{H^{s}}^{2})
+\|u\|_{H^{s+\frac{\alpha}{2}}}^{2}+\|\theta\|_{H^{s+\frac{\beta}{2}}}^{2}\nonumber\\
&\leq& C(1+\|\omega\|_{L^{\infty}}+\|\nabla \theta\|_{L^{\infty}})\ln(e+\|
u\|_{H^{s}}+\|\theta\|_{H^{s}}^{2})(\|{u}\|_{H^{s}}^{2}+\|\theta\|_{H^{s}}^{2}).
\end{eqnarray}
Applying the log-Gronwall type inequality  as well as the estimates (\ref{tNew007}) and (\ref{Vorticity}), we eventually obtain the desired estimates.
This concludes the proof of Theorem \ref{Th1}.
\end{proof}

\vskip .3in
\appendix
\section{The proof of Lemmas \ref{NCE} and \ref{Lem23}}
Before proving Lemmas \ref{NCE} and \ref{Lem23}, we first recall the
so-called Littlewood-Paley operators and their elementary properties
which allow us to define the Besov spaces (see for example
\cite{BCD,BL,MWZ2012,Tri}). It will be also convenient to introduce some function
spaces and review some well-known facts.

Let $(\chi, \varphi)$ be a couple of smooth
functions with values in $[0, 1]$ such that $\chi\in
C_{0}^{\infty}(\mathbb{R}^{n})$ is supported in the ball
$\mathcal{B}:= \{\xi\in \mathbb{R}^{n}, |\xi|\leq \frac{4}{3}\}$,
$\rm{\varphi\in C_{0}^{\infty}(\mathbb{R}^{n})}$ is supported in the annulus
$\mathcal{C}:= \{\xi\in \mathbb{R}^{n}, \frac{3}{4}\leq |\xi|\leq
\frac{8}{3}\}$ and satisfy
$$\chi(\xi)+\sum_{j\in \mathbb{N}}\varphi(2^{-j}\xi)=1, \quad  \forall \xi\in \mathbb{R}^{n},\quad\sum_{j\in \mathbb{Z}}\varphi(2^{-j}\xi)=1,  \  \forall  \xi \in \mathbb{R}^{n}\setminus \{0 \}.$$
For every $ u\in S'$ (tempered distributions)  we define the
non-homogeneous Littlewood-Paley operators as follows,
$$\Delta_{j}u=0\ j\leq -2; \ \quad \Delta_{-1}u=\chi(D)u; \ \quad\forall j\in \mathbb{N}, \ \quad\Delta_{j}u=\varphi(2^{-j}D)u.
$$We shall also denote
$$\ S_{j}u:=\sum_{-1\leq k\leq j-1} \Delta_{k}u,\qquad \widetilde{\Delta}_{j}
u:=\Delta_{j-1}u+\Delta_{j}u+\Delta_{j+1}u.$$
We now point
out several simple facts concerning the operators $\Delta_{j}$: By
compactness of the supports of the series of Fourier transform, we
have $$\Delta_{j}\Delta_{l}u\equiv0, \quad |j-l|\geq2 \quad and
\quad \Delta_{k}(S_{l}u\Delta_{l}v)\equiv0 \quad |k-l|\geq5.$$ for
any $u$ and $v$.
Moreover, it is easy to check that
$$\mbox{supp}\,  \mathcal{F}\big( {S}_{j-1}u {\Delta}_{j}v \big)\approx
\Big\{\xi\,\,|\,\, \frac{1}{12}2^{j}\leq |\xi|\leq
\frac{10}{3}2^{j}\Big\},$$
$$\mbox{supp}\, \mathcal{F}\big(\widetilde{ {\Delta}}_{j}u {\Delta}_{j}v \big)
\subset \Big\{\xi\,\,|\,\, |\xi|\leq 8 \times 2^{j}\Big\},$$ where
$\mathcal{F}$ denotes the Fourier transform and $A\approx B$ to
denote $C^{-1} B\leq A\leq C B$ for some positive
constant $C$.\\
Let us recall Let us recall the definition of homogeneous and
inhomogeneous Besov spaces through the dyadic decomposition.
\begin{define}
Let $s\in \mathbb{R}, (p,r)\in[1,+\infty]^{2}$. The inhomogeneous
Besov space $B_{p,r}^{s}$ are defined as a space of $f\in
S'(\mathbb{R}^{n})$ such that
$$ B_{p,r}^{s}=\{f\in S'(\mathbb{R}^{n}),  \|f\|_{B_{p,r}^{s}}<\infty\},$$
where
\begin{equation}\label{1}\nonumber
 \|f\|_{B_{p,r}^{s}}=\left\{\aligned
&\Big(\sum_{j\geq-1}2^{jrs}\|\Delta_{j}f\|_{L^{p}}^{r}\Big)^{\frac{1}{r}}, \quad \forall \ r<\infty,\\
&\sup_{j\geq-1}
2^{js}\|\Delta_{j}f\|_{L^{p}}, \quad \forall \ r=\infty.\\
\endaligned\right.
\end{equation}
\end{define}
Many frequently used function spaces are special cases of Besov spaces.
For $s\in \mathbb{R},(p,r)\in[1,+\infty]^{2}$, we have the following fact
$$\|f\|_{ {B}_{2,2}^{s}}\approx\|f\|_{H^{s}}.$$
For any $s\in \mathbb{R}$ and $1<q<\infty$,
$$
B^{s}_{q,\min\{q,2\}} \hookrightarrow W^{s,\,q} \hookrightarrow B^{s}_{q,\max\{q,2\}}.
$$

Bernstein inequalities are fundamental in the analysis involving
Besov spaces and these inequalities trade integrability for
derivatives.
\begin{lemma} [see \cite{BCD}]\label{lem22}
Let $k\in \mathbb{N}\cup\{0\}, 1\leq a\leq b\leq\infty$. Assume
$k=|\alpha|$, then there exist positive constants $C_1$ and $C_2$
independent of $j$ and $f$ only such that
$$
\mbox{supp}\, \widehat{f} \subset \{\xi\in \mathbb{R}^n: \,\, |\xi|
\lesssim  2^j \} \Rightarrow \|\partial^{\alpha} f\|_{L^b} \le C_1\,
2^{j k  + jn(\frac{1}{a}-\frac{1}{b})} \|f\|_{L^a};
$$
$$
\mbox{supp}\, \widehat{f} \subset \{\xi\in \mathbb{R}^n: \,\,|\xi|
\thickapprox 2^j \} \Rightarrow C_1\, 2^{ j k} \|f\|_{L^b } \le
\|\partial^{\alpha} f\|_{L^b } \le C_2\, 2^{  j k + j
n(\frac{1}{a}-\frac{1}{b})} \|f\|_{L^a}.
$$
Here we use $A\lesssim B$ to denote $A\leq C B$ for some positive
constant $C$.
\end{lemma}

To prove Lemmas \ref{NCE} and \ref{Lem23}, the following lemma will
be used extensively.
\begin{Pros}\label{ProsA1}
Given $(p_{1},\,p_{2})\in [2, \infty]^{2}$ and $p\in[2,\,\infty)$
such that $\frac{1}{p}=\frac{1}{p_{1}}+\frac{1}{p_{2}}$. Let f, g
and h be three functions such that $\nabla f\in L^{p_{1}}$, $g\in
L^{p_{2}}$ and $xh\in L^{1}$. Then it holds
$$\|h\star(fg)-f(h\star g)\|_{L^{p}}\leq \|xh\|_{L^{1}}
\|\nabla f\|_{L^{p_{1}}}\|g\|_{L^{p_{2}}},$$ where $\star$ stands
for the convolution symbol.
\end{Pros}
\begin{proof}[\textbf{Proof of Proposition \ref{ProsA1}}]
We remark that Proposition \ref{ProsA1} with $p_{1}=p $ and $ p_{2}=\infty$ has been proven in \cite{HK4}. The interested reader may refer to \cite{JMWZ} for general case. To facilitate the reader, we give the details.

By direct calculation, one may easily show that
\begin{eqnarray}
h\star(fg)(x)-f(h\star g)(x)&=&\int_{\mathbb{R}^{2}}{h(x-y)g(y)(f(y)-f(x))\,dy}\nonumber\\
&=&\int_{\mathbb{R}^{2}}\int_{0}^{1}{h(x-y)g(y)(y-x).(\nabla f)(x+(y-x)t)\,dy dt}\nonumber\\
&=&\int_{\mathbb{R}^{2}}\int_{0}^{1}{h\Big(\frac{z}{t}\Big)g\Big(x-\frac{z}{t}\Big)\frac{z}{t}.(\nabla f)(x-z)\frac{1}{t^{2}}\,dz dt}.\nonumber
\end{eqnarray}
According to the Minkowski inequality and  the H\"{o}lder inequality, one has
\begin{eqnarray}
\|h\star(fg)-f(h\star g)\|_{L^{p}}&\leq&\int_{\mathbb{R}^{2}}
\int_{0}^{1}{h\Big(\frac{z}{t}\Big)\frac{z}{t^{3}}\|\nabla f\|_{L_{x}^{p_{1}}}\|g\|_{L_{x}^{p_{2}}} \,d z dt}\nonumber\\
&\leq& \|xh\|_{L^{1}}
\|\nabla f\|_{L^{p_{1}}}\|g\|_{L^{p_{2}}},
\end{eqnarray}
which is nothing but the desired result.
\end{proof}

Now let us proceed to prove Lemma \ref{NCE}.
To start, we use Bony's decomposition to present the commutator as
\begin{eqnarray}\label{A001}
\Delta_{k}[\Lambda^{\delta},f]g&=&
\sum_{|j-k|\leq 4}\Delta_{k}\Big([\Lambda^{\delta},\,S_{j-1}f]
\Delta_{j}g\Big)+\sum_{|j-k|\leq 4}\Delta_{k}
\Big([\Lambda^{\delta},\,\Delta_{j}f]S_{j-1}g\Big)\nonumber\\&&+
\sum_{j-k \geq -4}\Delta_{k}\Big([\Lambda^{\delta},\,\Delta_{j}f]\widetilde
{\Delta}_{j}g\Big)\nonumber\\&:=&
N_{1}+N_{2}+N_{3}.
\end{eqnarray}
Now we recall the following fact. Let $\mathcal{A}$ bet an annulus centered at the origin. Then for every $F$ with spectrum supported on $2^{j}\mathcal{A}$, there exists $\eta\in \mathcal{S}(\mathbb{R}^{n})$ whose Fourier transform supported away from the origin, such that
$$\Lambda^{\delta}F=2^{j(n+\delta)}\eta(2^{j}.)\star F.$$
For fixed $k$, the summation
over $|j-k|\leq 4$ involves only a finite number of $j's$. For the sake of brevity, we shall
replace the summations by their representative term with $j=k$ in $N_{1}$ and $N_{2}$.
In view of the above fact, Berstein's lemma and Proposition \ref{ProsA1}, we thus get
\begin{eqnarray}\label{A002}\|N_{1}\|_{L^{p}}
&\leq&C\|x 2^{k(n+\delta)}\eta(2^{k}x)\|_{L^{1}}\|\nabla
S_{k-1}f\|_{L^{p_{1}}} \|\Delta_{k}g\|_{L^{p_{2}}}
\nonumber\\&\leq&C2^{k(\delta-1)}\|\nabla f\|_{L^{p_{1}}}
\|\Delta_{k}g\|_{L^{p_{2}}}.
\end{eqnarray}
Similarly, one can also deduce that
\begin{eqnarray}\label{A003}\|N_{2}\|_{L^{p}}
&\leq&C\|x 2^{k(n+\delta)}\eta(2^{k}x)\|_{L^{1}}\|\Delta_{k}\nabla
f\|_{L^{p_{1}}} \|S_{k-1}g\|_{L^{p_{2}}} \nonumber\\&\leq&
C2^{k(\delta-1)}\|\nabla f\|_{L^{p_{1}}} \sum_{l\leq
k-2}\|\Delta_{l}g\|_{L^{p_{2}}} \nonumber\\&\leq& C\|\nabla
f\|_{L^{p_{1}}} \sum_{l\leq
k-2}2^{(k-l)(\delta-1)}2^{l(\delta-1)}\|\Delta_{l}g\|_{L^{p_{2}}}.
\end{eqnarray}
Finally, the last term $N_{3}$ can be rewritten as
\begin{eqnarray}\label{A004}N_{3}&=&\sum_{j-k \geq -4}\Delta_{k}
\Big(\Lambda^{\delta}(\Delta_{j}f
\widetilde{\Delta}_{j}g)-\Delta_{j}f
\Lambda^{\delta}\widetilde{\Delta}_{j}g\Big)\nonumber\\
&=&\sum_{j-k \geq -4,\,\,j\geq0}\Delta_{k}
\Big(\Lambda^{\delta}(\Delta_{j}f
\widetilde{\Delta}_{j}g)-\Delta_{j}f
\Lambda^{\delta}\widetilde{\Delta}_{j}g\Big)\nonumber\\&&+
\sum_{-1-k \geq -4}\Delta_{k} \Big(\Lambda^{\delta}(\Delta_{-1}f
\widetilde{\Delta}_{-1}g)-\Delta_{-1}f
\Lambda^{\delta}\widetilde{\Delta}_{-1}g\Big)\nonumber\\
&:=&N_{31}+N_{32}
\end{eqnarray}
By Berstein's lemma, the term $N_{31}$ can be bounded without using
commutator structure
\begin{eqnarray}\label{A005}\|N_{31}\|_{L^{p}}
&\leq& C\sum_{j-k \geq -4,\,\,j\geq0}\Big(\Big\|\Delta_{k}
\big(\Lambda^{\delta}(\Delta_{j}f
\widetilde{\Delta}_{j}g)\big)\Big\|_{L^{p}}+\Big\|\Delta_{k}
\big(\Delta_{j}f
\Lambda^{\delta}\widetilde{\Delta}_{j}g\big)\Big\|_{L^{p}}\Big)
\nonumber\\
&\leq& C\sum_{j-k \geq
-4,\,\,j\geq0}\Big(\Big\|\Lambda^{\delta}(\Delta_{j}f
\widetilde{\Delta}_{j}g)\Big\|_{L^{p}}+\Big\|\Delta_{j}f
\Lambda^{\delta}\widetilde{\Delta}_{j}g\Big\|_{L^{p}}\Big)
\nonumber\\
&\leq& C\sum_{j-k \geq
-4,\,\,j\geq0}2^{j(\delta-1)}\|\Delta_{j}\nabla f\|_{L^{p_{1}}}
\|\Delta_{j}g\|_{L^{p_{2}}}
\nonumber\\
&\leq& C\sum_{j-k \geq -4}2^{j(\delta-1)}\|\nabla f\|_{L^{p_{1}}}
\|\Delta_{j}g\|_{L^{p_{2}}}.
\end{eqnarray}
Resorting Berstein's lemma again, the term $N_{32}$ admits the
following bound
\begin{eqnarray}\|N_{32}\|_{L^{p}}&\leq&
\sum_{-1\leq k \leq 3}\|\Delta_{k} \big(\Lambda^{\delta}(\Delta_{-1}f
\widetilde{\Delta}_{-1}g)\big)\|_{L^{p}}+\sum_{-1\leq k \leq 3}\|\Delta_{k} \big(\Delta_{-1}f
\Lambda^{\delta}\widetilde{\Delta}_{-1}g\big)\|_{L^{p}}
\nonumber\\
&\leq& C\sum_{-1\leq k \leq 3}(2^{k\delta}+1)\|\Delta_{-1}f\|_{L^{2p}}\|
\widetilde{\Delta}_{-1}g\|_{L^{2p}}
\nonumber\\
&\leq& C\chi_{\{-1\leq k \leq 3\}}\|f\|_{L^{2}}\|g\|_{L^{2}}.
\end{eqnarray}
By the definition of $B_{p,r}^{s}$, we have
\begin{eqnarray}\label{A006}
\|[\Lambda^{\delta},f]g\|_{B_{p,r}^{s}} &\leq& \big\|2^{ks}\|N_{1}\|_{L^{p}}\big\|_{l_{k}^{r}}+\big\|2^{ks}\|N_{2}\|_{L^{p}}\big\|_{l_{k}^{r}}
+\big\|2^{ks}\|N_{3}\|_{L^{p}}\big\|_{l_{k}^{r}}\nonumber\\
&\leq& C\|\nabla f\|_{L^{p_{1}}}\big\|2^{k(s+\delta-1)}
\|\Delta_{k}g\|_{L^{p_{2}}}\big\|_{l_{k}^{r}}\nonumber\\&&+C\|\nabla
f\|_{L^{p_{1}}}\Big\|\sum_{l\leq
k-2}2^{(k-l)(s+\delta-1)}2^{l(s+\delta-1)}\|\Delta_{l}g\|_{L^{p_{2}}}\Big\|_{l_{k}^{r}}
\quad (s+\delta-1<0)\nonumber\\&&+C\|\nabla f\|_{L^{p_{1}}}\Big
\|\sum_{j-k \geq -4}2^{(k-j)s}2^{j(s+\delta-1)}
\|\Delta_{j}g\|_{L^{p_{2}}} \Big\|_{l_{k}^{r}}\quad
(s>0)\nonumber\\&&+C \|f\|_{L^{2}}\|g\|_{L^{2}}
\nonumber\\
&\leq& C(p,r,\delta,s)(\|\nabla f\|_{L^{p_{1}}}\|g\|_{
{B}_{p_{2},r}^{s+\delta-1}}+\|f\|_{L^{2}}\|g\|_{L^{2}}).
\end{eqnarray}
Therefore, this concludes the proof of Lemma \ref{NCE}.

\vskip .2in
The proof of Lemma \ref{Lem23} is very similar to that of Lemma
\ref{NCE}. Indeed, we can regard the operator $\mathcal {R}_{\beta}$
as the operator $\Lambda^{1-\beta}$ without any difference. Now we
just view $f=u$, $g=\nabla\theta$ and $\delta=1-\beta$.
 Therefore, we have
\begin{eqnarray}
\Delta_{k}[\mathcal {R}_{\beta},\,u\cdot\nabla]\theta&=&
\sum_{|j-k|\leq 4}\Delta_{k}\Big([\mathcal
{R}_{\beta},\,S_{j-1}u\cdot\nabla]
\Delta_{j}\theta\Big)+\sum_{|j-k|\leq 4}\Delta_{k} \Big([\mathcal
{R}_{\beta},\,\Delta_{j}u\cdot\nabla]S_{j-1}\theta\Big)\nonumber\\&&+
\sum_{j-k \geq -4}\Delta_{k}\Big([\mathcal
{R}_{\beta},\,\Delta_{j}u\cdot\nabla]\widetilde
{\Delta}_{j}\theta\Big)\nonumber\\&:=&
\widetilde{N}_{1}+\widetilde{N}_{2}+\widetilde{N}_{3}.\nonumber
\end{eqnarray}
The same as $N_{1}$ and $N_{2}$, we can conclude
\begin{eqnarray}\|\widetilde{N}_{1}\|_{L^{p}}
\leq C2^{k(1-\beta)}\|\nabla u\|_{L^{p_{1}}}
\|\Delta_{k}\theta\|_{L^{p_{2}}}\nonumber
\end{eqnarray}
and
\begin{eqnarray}\|\widetilde{N}_{2}\|_{L^{p}}
\leq C\|\nabla u\|_{L^{p_{1}}} \sum_{l\leq
k-2}2^{(l-k)\beta}2^{l(1-\beta)}\|\Delta_{l}\theta\|_{L^{p_{2}}}.\nonumber
\end{eqnarray}
However, we need to deal with the term $\widetilde{N}_{3}$
differently. We rewrite $\widetilde{N}_{3}$ as
\begin{eqnarray}\widetilde{N}_{3}&=&\sum_{j-k \geq -4}\Delta_{k}\Big([\mathcal
{R}_{\beta},\,\Delta_{j}u\cdot\nabla]\widetilde
{\Delta}_{j}\theta\Big)\nonumber\\&=& \sum_{j-k \geq
-4,\,j\geq0}\Delta_{k}\Big(\mathcal
{R}_{\beta}(\Delta_{j}u\cdot\nabla\widetilde
{\Delta}_{j}\theta)-\Delta_{j}u\cdot\nabla\mathcal
{R}_{\beta}\widetilde
{\Delta}_{j}\theta\Big)\nonumber\\&&-\sum_{-1-k \geq
-4}\Delta_{k}\Big(\mathcal
{R}_{\beta}(\Delta_{-1}u\cdot\nabla\widetilde
{\Delta}_{-1}\theta)-\Delta_{-1}u\cdot\nabla\mathcal
{R}_{\beta}\widetilde {\Delta}_{-1}\theta\Big)\nonumber\\
&:=&\widetilde{N}_{31}+\widetilde{N}_{32}.\nonumber
\end{eqnarray}
By Berstein's lemma and the divergence-free condition
\begin{eqnarray}\|\widetilde{N}_{31}\|_{L^{p}}
&\leq& C\sum_{j-k \geq -4,\,\,j\geq0}\Big(\Big\|\Delta_{k}\big(\mathcal
{R}_{\beta}(\Delta_{j}u\cdot\nabla\widetilde
{\Delta}_{j}\theta)\big)\Big\|_{L^{p}}+\Big\|\Delta_{k}\big(\Delta_{j}u\cdot\nabla\mathcal
{R}_{\beta}\widetilde
{\Delta}_{j}\theta\big)\Big\|_{L^{p}}\Big)
\nonumber\\
&\leq& C\sum_{j-k \geq
-4,\,\,j\geq0}\Big(\Big\|\Delta_{k}\big(\mathcal
{R}_{\beta}\nabla\cdot(\Delta_{j}u \widetilde
{\Delta}_{j}\theta)\big)\Big\|_{L^{p}}+\Big\|\Delta_{k}\nabla\cdot\big(\Delta_{j}u \mathcal
{R}_{\beta}\widetilde
{\Delta}_{j}\theta\big)\Big\|_{L^{p}}\Big)
\nonumber\\
&\leq& C\sum_{j-k \geq
-4,\,\,j\geq0}2^{(k-j)(2-\beta)}\| \nabla u\|_{L^{p_{1}}}
2^{(1-\beta)j}\|\Delta_{j}\theta\|_{L^{p_{2}}}\nonumber\\&&+C\sum_{j-k \geq
-4,\,\,j\geq0}2^{k-j}\| \nabla u\|_{L^{p_{1}}}
2^{(1-\beta)j}\|\Delta_{j}\theta\|_{L^{p_{2}}},\nonumber
\end{eqnarray}
and
\begin{eqnarray}\|\widetilde{N}_{32}\|_{L^{p}}
\leq C\chi_{\{-1\leq k \leq 3\}}\|u\|_{L^{r}}\|\theta\|_{L^{2}}.
\end{eqnarray}
Putting all the above estimates together and making use of the Young inequality for series convolution yield
\begin{eqnarray}
\|[\mathcal {R}_{\beta},\,u\cdot\nabla]\theta\|_{L^{p}}&\leq& \sum_{k\geq-1}\|\Delta_{k}[\mathcal {R}_{\beta},\,u\cdot\nabla]\theta\|_{L^{p}}
\nonumber\\
&\leq& C\|\nabla u\|_{L^{p_{1}}}\sum_{k\geq-1}2^{k(1-\beta)}
\|\Delta_{k}\theta\|_{L^{p_{2}}}\nonumber\\&&+C\|\nabla u\|_{L^{p_{1}}} \sum_{k\geq-1}\sum_{l\leq
k-2}2^{(l-k)\beta}2^{l(1-\beta)}\|\Delta_{l}\theta\|_{L^{p_{2}}}
\nonumber\\&&+
C\| \nabla u\|_{L^{p_{1}}}\sum_{k\geq-1}\sum_{j-k \geq
-4,\,\,j\geq0}2^{(k-j)(2-\beta)}
2^{(1-\beta)j}\|\Delta_{j}\theta\|_{L^{p_{2}}}\nonumber\\&&+C\| \nabla u\|_{L^{p_{1}}}\sum_{k\geq-1}\sum_{j-k \geq
-4,\,\,j\geq0}2^{k-j}
2^{(1-\beta)j}\|\Delta_{j}\theta\|_{L^{p_{2}}}\nonumber\\&&+C\sum_{k\geq-1}\chi_{\{-1\leq k \leq 3\}}\|u\|_{L^{r}}\|\theta\|_{L^{2}}
\nonumber\\
&\leq&
C\|\nabla
u\|_{L^{p_{1}}}\|\theta\|_{B_{p_{2},1}^{1-\beta}}+
\|u\|_{L^{r}}\|\theta\|_{L^{2}}.
\end{eqnarray}
This completes the proof of Lemma \ref{Lem23}.

\vskip .3in
\section*{Acknowledgements}
Both authors would like to thank Dr. Liutang Xue for his valuable comments and stimulating discussions. Special thanks also go to Prof. Jiahong Wu for his interest and kind suggestions.


\vskip .4in
\begin{thebibliography}{00} \frenchspacing
\bibitem{ACSWXY}
D. Adhikari, C. Cao, H. Shang, J. Wu, X. Xu, Z. Ye, Global regularity results for the 2D Boussinesq equations with partial dissipation, J. Differential Equations 260 (2016), 1893-1917.
\bibitem{ACWX}
D. Adhikari, C. Cao, J. Wu, X. Xu, Small global solutions to the damped two-dimensional Boussinesq equations, J. Differential Equations  256 (2014), 3594-3613.
\bibitem{BCD}
H. Bahouri, J.-Y. Chemin, R. Danchin: Fourier Analysis and Nonlinear
Partial Differential Equations, Grundlehren der mathematischen
Wissenschaften, 343, Springer (2011).
\bibitem{BL}
J. Bergh and J. L\"{o}fstr\"{o}m, Interpolation Spaces, An Introduction, Springer-Verlag, Berlin-Heidelberg-New York, 1976.
\bibitem{BKM}
J.T. Beale, T. Kato, A. Majda, Remarks on the breakdown of smooth solutions for the 3-D Euler equations, Comm. Math. Phys. 94 (1984), 61-66.
\bibitem{BG}
H. Br$\rm \acute{e}$zis, T. Gallouet, Nonlinear Schr$\rm \ddot{o}$dinger evolution equations, Nonlinear Anal. 4(4) (1980), 677-681.
\bibitem{Can}
J. Cannon, E. DiBenedetto, The initial value problem for the Boussinesq equation
with data in $L^{p}$, Lecture Notes in Mathematics, Vol. 771. Springer, Berlin, (1980), 129-144.
\bibitem{CW}
C. Cao, J. Wu, Global regularity for the 2D anisotropic Boussinesq
equations with vertical dissipation, Arch. Rational Mech 208 (2013),
985-1004.
\bibitem{Cap}
M. Caputo, Linear models of dissipation whose Q is almost frequency independent-II, Geophy. J. R. Astr. Soc. 13 (1967), 529-539.
\bibitem{C1}
D. Chae, Global regularity for the 2D Boussinesq equations with
partial viscosity terms, Adv. Math. 203 (2006), 497-513.
\bibitem{CKN}
D. Chae, S. Kim, H. Nam, Local existence and blow-up criterion of H$\rm\ddot{o}$lder continuous solutions of the Boussinesq equations, Nagoya Math. J., 155 (1999), 55-80.
\bibitem{CW2}
D. Chae, J. Wu, The 2D Boussinesq equations with logarithmically
supercritical velocities, Adv. Math. 230 (2012), 1618-1645.
\bibitem{CV}
P. Constantin, V. Vicol, Nonlinear maximum principles for dissipative linear nonlocal
operators and applications, Geom. Funct. Anal. 22 (2012), 1289-1321.
\bibitem{D}
R. Danchin, Remarks on the lifespan of the solutions to some models
of incompressible fluid mechanics, Proc. Amer. Math. Soc. 141
(2013), 1979-1993.
\bibitem{DP2}
R. Danchin, M. Paicu, Global well-posedness issues for the inviscid Boussinesq system with Yudovich's type data, Comm. Math. Phys. 290 (2009), 1-14.
\bibitem{DP3}
R. Danchin, M. Paicu, Global existence results for the anisotropic
Boussinesq system in dimension two, Math. Models Methods Appl. Sci.
21 (2011), 421-457.
\bibitem{DImbert}
J. Droniou, C. Imbert, Fractal first order partial differential equations, Arch. Ration. Mech. Anal. 182 (2006), 299-331.
\bibitem{Gill}
A.E. Gill, Atmosphere-Ocean Dynamics, Academic Press, London, 1982.
\bibitem{HMOW}
H. Hajaiej, L. Molinet, T. Ozawa, B. Wang, Sufficient and necessary
conditions for the fractional Gagliardo-Nirenberg inequalities and
applications to Navier-Stokes and generalized Boson equations, in:
T. Ozawa, M. Sugimoto (Eds.), RIMS Kkyroku Bessatsu B26: Harmonic
Analysis and Nonlinear Partial Differential Equations, Vol. 5, 2011,
pp. 159-175.
\bibitem{Hmidi2011}
T. Hmidi, On a maximum principle and its application to the logarithmically critical Boussinesq system, Anal. PDE (2011), 247-284.
\bibitem{HK3}
T. Hmidi, S. Keraani, F. Rousset, Global well-posedness for a
Boussinesq-Navier-Stokes system with critical dissipation, J.
Differential Equations 249 (2010), 2147-2174.
\bibitem{HK4}
T. Hmidi, S. Keraani, F. Rousset, Global well-posedness for
Euler-Boussinesq system with critical dissipation, Comm. Partial
Differential Equations 36 (2011), 420-445.
\bibitem{HL}
T. Y. Hou, C. Li, Global well-posedness of the viscous Boussinesq
equations, Discrete Contin. Dyn. Syst. 12 (2005),
1-12.
\bibitem{JMWZ}
Q. Jiu, C. Miao, J. Wu, Z. Zhang, The 2D incompressible Boussinesq
equations with general critical dissipation,  SIAM J. Math. Anal. 46 (2014), 3426-3454.
\bibitem{JWYang}
Q. Jiu, J. Wu, W. Yang, Eventual regularity of the two-dimensional
Boussinesq equations with supercritical dissipation, J. Nonlinear Science 25 (2015), 37-58.
\bibitem{KRTW}
D. KC, D. Regmi, L. Tao, J. Wu, The 2D Euler-Boussinesq equations with a singular velocity, J. Differential Equations  257 (2014), 82-108.
\bibitem{LaiPan}
M. Lai, R. Pan, K. Zhao, Initial boundary value problem for two-dimensional viscous Boussinesq equations,  Arch. Ration. Mech. Anal.  199 (2011), 739-760.
\bibitem{LLT}
A. Larios, E. Lunasin, E.S. Titi, Global well-posedness for the 2D Boussinesq system with anisotropic viscosity and without heat diffusion, J. Differential
Equations 255 (2013), 2636-2654.
\bibitem{LT}
J. Li, E.S. Titi, Global well-posedness of the 2D Boussinesq
equations with vertical dissipation, arXiv:1502.06180.
\bibitem{LSWXY}
J. Li, H. Shang, J. Wu, X. Xu, Z. Ye, Regularity criteria for the 2D
Boussinesq equations with supercritical dissipation, in press.
\bibitem{LWZ}
X. Liu, M. Wang, Z. Zhang, Local well-posedness and blowup criterion of the Boussinesq equations in critical Besov spaces, J. Math. Fluid Mech. 12 (2010), 280-292.
\bibitem{MB}
 A. Majda, A. Bertozzi, Vorticity and Incompressible
Flow, Cambridge University Press, Cambridge, 2001.
\bibitem{MWZ2012}
C. Miao, J. Wu and Z. Zhang, Littlewood-Paley Theory and its Applications in Partial Differential Equations of Fluid Dynamics, Science Press, Beijing, China, 2012 (in Chinese).
\bibitem{MX}
C. Miao, L. Xue, On the global well-posedness of a class of
Boussinesq-Navier-Stokes systems, NoDEA Nonlinear Differential
Equations Appl. 18 (2011), 707-735.
\bibitem{PG}
J. Pedlosky, Geophysical fluid dynamics, New York, Springer-Verlag,
1987.
\bibitem{PSW}
A. Pekalski, K. Sznajd-Weron, Anomalous Diffusion. From Basics to Applications, Lecture Notes in Phys., vol. 519,
Springer-Verlag, Berlin, 1999.
\bibitem{Sil20121}
L. Silvestre, On the differentiability of the solution to an
equation with drift and fractional diffusion, Indiana Univ. Math. J.
61 (2012), 557-584.
\bibitem{SW}
A. Stefanov, J. Wu, A global regularity result for the 2D Boussinesq equations with critical dissipation, arXiv:1411.1362v3 math.AP.
\bibitem{Tri} H. Triebel, Theory of Function Spaces II, Birkhauser Verlag, 1992.
\bibitem{WuXue2012}
G. Wu, L. Xue, Global well-posedness for the 2D inviscid B$\rm\acute{e}$nard system with fractional diffusivity and Yudovich's type data, J. Differential Equations 253 (2012), 100-125.
\bibitem{WuXu}
J. Wu, X. Xu, Well-posedness and inviscid limits of the Boussinesq equations with fractional Laplacian dissipation,  Nonlinearity  (2014), 2215--2232.
\bibitem{WXY}
J. Wu, X. Xu, Z. Ye, Global smooth solutions to the n-dimensional damped models of incompressible fluid mechanics with small initial datum, J. Nonlinear Science 25 (2015), 157-192.
\bibitem{XX}
X. Xu, Global regularity of solutions of 2D Boussinesq equations
with fractional diffusion,  Nonlinear Anal. 72 (2010), 677-681.
\bibitem{XX2014}
X. Xu, L. Xue, Yudovich type solution for the 2D inviscid Boussinesq system with critical and supercritical dissipation, J. Differential Equations 256 (2014), 3179-3207.
\bibitem{YJW}
W. Yang, Q. Jiu, J. Wu, Global well-posedness for a class of 2D Boussinesq systems with fractional dissipation, J. Differential Equations 257 (2014), 4188-4213.
\bibitem{YZ}
Z. Ye, Blow-up criterion of smooth solutions for the
Boussinesq equations, Nonlinear Anal. 110 (2014), 97-103.
\bibitem{YX2015}
Z. Ye, X. Xu, Remarks on global regularity of the 2D Boussinesq
equations with fractional dissipation, Nonlinear Anal. 125 (2015),
715-724.
\bibitem{YXX}
Z. Ye, X. Xu, L. Xue, On the global regularity of the 2D Boussinesq equations with fractional dissipation, submitted for publication (2014).
\end{thebibliography}
\end{document}